\documentclass[10pt,a4paper,reqno]{amsart} 
\usepackage[final]{optional}
\usepackage[british,american]{babel}
\usepackage[margin=0cm]{geometry}
\usepackage{lipsum}
\usepackage[leqno]{amsmath}
\usepackage{mathrsfs}

\usepackage{amsfonts, amsmath, wasysym, caption}

\usepackage{amssymb,amsthm, paralist}

\usepackage{amsopn}
\usepackage{accents,bbm,
dsfont,esint,url,verbatim}
\usepackage[normalem]{ulem}
\usepackage{bbold}

\usepackage{
latexsym,
nicefrac
}
\newcommand\Item[1][]{%
  \ifx\relax#1\relax  \item \else \item[#1] \fi
  \abovedisplayskip=0pt\abovedisplayshortskip=0pt~\vspace*{-\baselineskip}}
  
\usepackage[usenames]{color}
\usepackage{tikz}
\usetikzlibrary{decorations.pathreplacing}
\usepackage{url}

\definecolor{darkgreen}{rgb}{0,0.5,0}
\definecolor{darkred}{rgb}{0.7,0,0}

\usepackage[colorlinks, 
citecolor=darkgreen, linkcolor=darkred
]{hyperref}

\usepackage{esint}
\usepackage[normalem]{ulem}

\usepackage{enumitem}
\theoremstyle{plain}
\newtheorem{lemma}{Lemma}[section]
\newtheorem{thm}[lemma]{Theorem}
\newtheorem{prop}[lemma]{Proposition}
\newtheorem{cor}[lemma]{Corollary}
\theoremstyle{definition}
\newtheorem{defn}[lemma]{Definition}

\theoremstyle{remark}
\newtheorem{remark}[lemma]{Remark}

\newtheorem{rmk}[lemma]{Remark}

\numberwithin{equation}{section}

\newcommand{\al}{\alpha}

\newcommand{\ga}{\gamma}

\newcommand{\de}{\delta}

\newcommand{\la}{\lambda}

\renewcommand{\th}{\theta}

\newcommand{\R}{\ensuremath{{\mathbb R}}}

\DeclareMathOperator{\pv}{PV}
\DeclareMathOperator{\ci}{Ci}

\newcommand{\beq}{\begin{equation}}
\newcommand{\eeq}{\end{equation}}
\newcommand{\beqs}{\begin{equation*}}
\newcommand{\eeqs}{\end{equation*}}
\newcommand{\beqa}{\begin{equation}\begin{aligned}}
\newcommand{\eeqa}{\end{aligned}\end{equation}}
\newcommand{\beqas}{\begin{equation*}\begin{aligned}}
\newcommand{\eeqas}{\end{aligned}\end{equation*}}

\newcommand{\half}{\frac{1}{2}}
\newcommand{\eps}{\varepsilon}

\DeclareMathOperator{\supp}{supp}

\title[Inviscid limit of the compressible Navier--Stokes equations]{{\sc
Inviscid limit of the compressible Navier--Stokes\\  equations for asymptotically isothermal gases
}
\\ 
}
\author[Matthew R. I. Schrecker]{Matthew R. I. Schrecker\textsuperscript{1}}\thanks{{}\textsuperscript{1}Department of Mathematics, University of Wisconsin Madison, Van Vleck Hall, 480
Lincoln Drive, Madison, Wisconsin 53706, USA. Email: schrecker@wisc.edu}  \author[Simon Schulz]{Simon Schulz\textsuperscript{2}}\thanks{{}\textsuperscript{2}Faculty of Mathematics, University of Cambridge,
Wilberforce Road,
Cambridge CB3 0WA, UK. Email: sms79@cam.ac.uk}

\begin{document}
\begin{abstract}
We prove the existence of
a relative finite-energy solution of the one-dimensional, isentropic Euler equations under the assumption of an asymptotically isothermal pressure law, that is, $p(\rho)/\rho = O(1)$ in the limit $\rho \to \infty$. This solution is obtained as the vanishing viscosity limit of classical solutions of the one-dimensional, isentropic, compressible Navier--Stokes equations. Our approach relies on the method of compensated compactness to pass to the limit rigorously in the nonlinear terms. Key to our strategy is the derivation of hyperbolic representation formulas for the entropy kernel and related quantities; among others, a special entropy pair used to obtain higher uniform integrability estimates on the approximate solutions. Intricate bounding procedures relying on these representation formulas then yield the required compactness of the entropy dissipation measures. In turn, we prove that the Young measure generated by the classical solutions of the Navier--Stokes equations reduces to a Dirac mass, from which we deduce the required convergence to a solution of the Euler equations.
\end{abstract}

\maketitle

\section{Introduction}
The dynamics of an inviscid barotropic fluid are modelled by the compressible Euler equations. In one spatial dimension, these read as
\begin{equation}\label{eq:euler}
\left\lbrace\begin{aligned}
& \rho_t + m_x = 0, \\
& m_t + ( \frac{m^2}{\rho} + p(\rho) )_x = 0,
\end{aligned}\right.
\end{equation}
where $(t,x) \in \mathbb{R}^2_+ = (0,\infty) \times \mathbb{R}$ are the variables of time and space, respectively, while $\rho \geq 0$ is the density of the fluid and $m= \rho u$ its momentum, with $u$ its velocity. The pressure $p(\rho)$ is a quantity that depends solely on the density. On the other hand, when one also takes into account the effects of viscosity, the equations governing the motion of the fluid are the one-dimensional, isentropic, compressible Navier--Stokes equations,
\begin{equation}\label{eq:ns}
\left\lbrace\begin{aligned}
& \rho^\varepsilon_t + (\rho^\varepsilon u^\varepsilon)_x = 0, \\
& (\rho^\varepsilon u^\varepsilon)_t + ( \rho^\varepsilon (u^\varepsilon)^2 + p(\rho^\varepsilon) )_x = \varepsilon u^\varepsilon_{xx},
\end{aligned}\right.
\end{equation}
where $\varepsilon > 0$ is the viscosity of the fluid. The work of Hoff \cite{Hoff} shows that, provided we supply appropriate initial data $(\rho_0,u_0) = (\rho,u)|_{t=0}$ and the pressure law $p(\rho)$ satisfies modest physical assumptions, there exists a unique regular classical solution of this latter system. It is then physically relevant (\textit{cf.}~\cite{Stokes}) to ask whether such solutions of the Navier--Stokes equations \eqref{eq:ns} converge to a solution of the Euler equations \eqref{eq:euler} in an appropriate topology. Moreover, the question of convergence of the vanishing viscosity limit is a crucial one for the well-posedness of the Euler equations, where admissibility criteria are conjectured to provide uniqueness of the weak solutions. This paper addresses this question of convergence and provides an affirmative answer, provided particular constitutive assumptions are satisfied by the pressure.

Much of the theory of gas dynamics revolves around so-called \emph{$\gamma$-law gases}, that is, fluids that incorporate a pressure law of the form $p(\rho) = \kappa \rho^\gamma$, where $\kappa > 0$ and $\gamma \geq 1$. When $\gamma \in (1,3)$, we say that the gas is \emph{polytropic}. When $\gamma=1$, the fluid is said to be \emph{isothermal}, in accordance with the ideal gas equation. Most real gases do not behave exactly as $\gamma$-law gases. However, many are approximated well by $\ga$-law models in a density regime near the vacuum and do exhibit isothermal-like growth for large values of the density (see, e.g.~Thorne \cite{Thorne}). To this end, we make the following definition.

\begin{defn}\label{def:asymp isothermal}
We say that a fluid behaves according to an \emph{asymptotically isothermal pressure law} if it satisfies the following constitutive assumptions.
\begin{enumerate}
    \item The pressure, $p \in C^1([0,\infty))\cap C^4((0,\infty))$, is such that $p(\rho) > 0$ for all $\rho > 0$, and satisfies 
    \begin{equation}\label{eq:st hyp gen non}
            p'(\rho) > 0 \quad \text{and} \quad \rho p''(\rho) + 2p'(\rho) > 0 \qquad \text{for all } \rho>0.
    \end{equation}
        \item There exist constants $\gamma \in (1,3)$ and $\kappa_1>0$, and a function $P \in C^4((0,\infty))$ such that
        \begin{equation}\label{eq:p asymp near 0 P}
            p(\rho) = \kappa_1 \rho^\gamma (1+ P(\rho)) \qquad \text{for } \rho \in [0,r), \end{equation}
            for some fixed $r > 0$, and there exists a positive $C_r$ such that $|P^{(j)}(\rho)|  \leq C_r \rho^{2\theta-j}$ for $\rho \in [0,r)$, and $j \in \{0,\dots,4\}$, where $\th=\frac{\ga-1}{2}$.
    \item There exists $\alpha > 0$, $\kappa_2 > 0$ and $C_p>0$ such that, for some fixed $R>0$,
    \begin{equation}\label{eq:p derivs}
        \left| \left( \frac{p(\rho)}{\rho} - \kappa_2 \right)^{(j)} \right| \leq C_p \rho^{-\alpha-j} \qquad \text{for } \rho \in [R,\infty) \text{ and } j \in \{0,\dots,4\}.
    \end{equation}
\end{enumerate}
\end{defn}

\begin{rmk}
Assumption \eqref{eq:st hyp gen non} ensures the strict hyperbolicity of the system and genuine non-linearity of the characteristic fields for $\rho>0$. Assumption \eqref{eq:p asymp near 0 P} guarantees the existence of the entropy kernel, see \eqref{eq:ent kernel} below. Indeed, an inspection of the results of \cite{ChenLeFloch,ChenLeFloch2} shows that the bound $\rho^{2\th-j}$ for $P^{(j)}$ may be replaced with $\rho^{(1+\de)\th-j}$ for any $\de>0$. However, for clarity of exposition, we choose not to make this assumption here.
\end{rmk}

Note that the class of gases considered in the previous definition is significantly wider than the one introduced in \cite{SS1}. Indeed, therein the authors considered the vanishing viscosity limit of solutions of \eqref{eq:ns} under the assumption of a pressure law satisfying the first two hypotheses of Definition \ref{def:asymp isothermal}, but where $p(\rho) = c_*\rho$ for $\rho \geq R$ instead of \eqref{eq:p derivs}. This precursor was called an \emph{approximately isothermal pressure law} in the doctoral thesis of the second author, \cite{thesis}. This assumption for large density allowed the authors to provide an explicit representation for entropies in this regime (\textit{cf.}~\cite[Theorem 2.6]{SS1}). Such a representation is no longer possible without the explicit form of the pressure.

The convergence of the vanishing physical viscosity limit from the Navier--Stokes equations to the Euler equations is a difficult problem with a rich history. Prior to the contribution of the authors in \cite{SS1}, Chen and Perepelitsa proved in \cite{ChenPerep1} that, for a $\gamma$-law gas of index $\gamma \in (1,\infty)$ and given any initial data of relative finite-energy, a solution of \eqref{eq:euler} could be obtained as the inviscid limit of solutions of \eqref{eq:ns}. Earlier than that, the existence of $L^\infty$ entropy solutions of the isentropic Euler equations had been obtained by means of vanishing artificial viscosities and from finite-difference schemes by DiPerna \cite{DiPerna2}, Chen \cite{Chen1}, Ding, Chen, and Luo \cite{DingChenLuo}, Lions, Perthame, and Tadmor \cite{LionsPerthameTadmor}, and Lions, Perthame, and Souganidis \cite{LionsPerthameSouganidis} for polytropic gases, by Chen and LeFloch \cite{ChenLeFloch,ChenLeFloch2} for general pressure laws, and by Huang and Wang \cite{HZ} for isothermal gases. In view of the fact that the Navier--Stokes equations do not admit natural invariant regions, this $L^\infty$ framework may not be applied directly, and so we work with the finite-energy framework, originally introduced by LeFloch and Westdickenberg in \cite{LeFlochWestdickenberg}, the notion of which we now develop.

Recall that an entropy pair is a pair of functions $(\eta,q):\mathbb{R}^2_+ \to \mathbb{R}^2$ such that
\begin{equation*}
\nabla q(\rho,m) = \nabla \eta(\rho,m) \nabla \begin{pmatrix}
m \\
\frac{m^2}{\rho} + p(\rho)
\end{pmatrix},
\end{equation*}
where $\nabla$ is the gradient with respect to the phase-space coordinates $(\rho,m)$. An important such entropy pair is the mechanical energy and its flux, $(\eta^*,q^*)$, given by the formulas
\begin{equation}\label{eq:physicalentropy}
\eta^*(\rho,m) = \frac{1}{2}\frac{m^2}{\rho} + \rho e(\rho), \qquad q^*(\rho,m) = \frac{1}{2}\frac{m^3}{\rho^2} + m e(\rho) + \rho m e'(\rho),
\end{equation}
where $e(\rho) := \int_0^\rho \frac{p(y)}{y^2} \, dy$ is the internal energy. In what follows, we will consider the energy of the solutions relative to nontrivial constant end-states $(\rho_\pm,u_\pm)$. Correspondingly, let $(\bar{\rho}(x),\bar{u}(x))$ be smooth, monotone functions such that, for some $L_0 > 1$,
\begin{equation*}
(\bar{\rho}(x),\bar{u}(x)) = \left\lbrace \begin{aligned}
&(\rho_+,u_+), \quad \text{for } x \geq L_0, \\
&(\rho_-,u_-), \quad \text{for } x \leq -L_0.
\end{aligned} \right.
\end{equation*}
These reference functions are fixed at this point and remain the same throughout the paper. We now define the relative mechanical energy with respect to $(\bar{\rho}(x),\bar{m}(x)) = (\bar{\rho}(x),\bar{\rho}(x)\bar{u}(x))$ as
\begin{equation*}
\begin{aligned}
\overline{\eta^*}(\rho,m) :=&\, \eta^*(\rho,m) - \eta^*(\bar{\rho},\bar{m}) - \nabla \eta^*(\bar{\rho},\bar{m}) \cdot (\rho-\bar{\rho},m-\bar{m}), \\
=& \,\frac{1}{2}\rho |u-\bar{u}|^2 + e^*(\rho,\bar{\rho}) \geq 0,
\end{aligned}
\end{equation*}
where $e^*(\rho,\bar{\rho}) = \rho e(\rho) - \bar{\rho} e(\bar{\rho}) - (\bar{\rho} e'(\bar{\rho}) + e(\bar{\rho}))(\rho-\bar{\rho}) \geq 0$. Then we define
\begin{equation}
E[\rho,u](t) := \int_\mathbb{R} \overline{\eta^*}(\rho,\rho u)(t,x) \, dx
\end{equation}
to be the total relative mechanical energy, relative to the end states $(\rho_\pm,u_\pm)$ and say that a pair $(\rho,m)$ with $m = \rho u$ is said to be of relative finite-energy if $E[\rho,u] < \infty$.

It is apparent from the definition of entropy pair, and the requirement that mixed partial derivatives commute, that any $C^2$ entropy function satisfies the \emph{entropy equation}, i.e.,
\begin{equation}
\eta_{\rho \rho} - \frac{p'(\rho)}{\rho^2} \eta_{uu} = 0.
\end{equation}
It is well known (see e.g.~\cite{ChenLeFloch,DiPerna2}) that any regular weak entropy (i.e.~one that vanishes at $\rho = 0$) may be generated by the integral of a test function $\psi \in C^2(\mathbb{R})$ against a fundamental solution $\chi(\rho,u,s)$ of the entropy equation, i.e.,
\begin{equation}\label{eq:generate entropy}
\eta^\psi(\rho,\rho u) = \int_\mathbb{R} \psi(s) \chi(\rho,u,s) \, ds,
\end{equation}
where this fundamental solution, the \emph{entropy kernel}, solves
\begin{equation}\label{eq:ent kernel}
\left\lbrace \begin{aligned}
& \chi_{\rho \rho} - k'(\rho)^2 \chi_{uu} = 0, \\
& \chi(0,u,s) = 0, \\
& \chi_\rho(0,u,s) = \delta_{u=s},
\end{aligned} \right.
\end{equation}
where \beq\label{def:k(rho)}
k(\rho) := \int_0^\rho \frac{\sqrt{p'(y)}}{y} \, dy.
\eeq The entropy kernel admits the Galilean invariance $\chi(\rho,u,s) = \chi(\rho,u-s,0) = \chi(\rho,0,s-u)$, and so we write (with slight abuse of notation) $\chi = \chi(\rho,u-s)$. The entropy kernel has corresponding flux $\sigma(\rho,u,s)$, the \emph{entropy flux kernel}, which behaves according to
\begin{equation}\label{eq:ent flux kernel}
\left\lbrace \begin{aligned}
& (\sigma-u\chi)_{\rho \rho} - k'(\rho)^2 (\sigma-u\chi)_{uu} = \frac{p''(\rho)}{\rho}\chi_u, \\
& (\sigma - u \chi)(0,u,s) = 0, \\
& (\sigma - u\chi)_\rho(0,u,s) = 0.
\end{aligned} \right.
\end{equation}
In what follows, we use the notation $h(\rho,u,s) = \sigma(\rho,u,s) - u \chi(\rho,u,s)$, and the same invariance $h = h(\rho,u-s)$ holds. One then generates the entropy flux $q^\psi$ corresponding to the entropy $\eta^\psi$ via 
\begin{equation}\label{eq:generate entropy flux}
q^\psi(\rho,\rho u) = \int_\mathbb{R} \psi(s) \sigma(\rho,u,s) \, ds.
\end{equation}

\begin{defn}\label{def:entropy sol}
Let $(\rho_0,u_0)$ be locally integrable initial data of relative finite-energy, that is, $(\rho_0,u_0) \in L^1_{loc}(\mathbb{R}^2_+)$ and $E[\rho_0,u_0] \leq E_0 < \infty$. We call $(\rho,u) \in L^1_{loc}(\mathbb{R}^2_+)$, with $\rho \geq 0$ almost everywhere, a \emph{relative finite-energy entropy solution} of the Euler equations \eqref{eq:euler} if:
\begin{enumerate}
    \item There exists a positive constant $M(E_0,t)$, increasing and continuous with respect to the variable $t$, such that
    \begin{equation*}
    E[\rho,u](t) \leq M(E_0,t) \qquad \text{for almost every } t \geq 0;
    \end{equation*}
    \item The pair $(\rho,u)$ solves the Cauchy problem \eqref{eq:euler} in the sense of distributions; 
    \item There exists a bounded Radon measure $\mu(t,x,s)$ on $\mathbb{R}^2_+\times\mathbb{R}$ such that
    \begin{equation*}
        \mu( U \times \mathbb{R} ) \leq 0 \qquad \text{for any open set } U \subset \mathbb{R}^2_+,
    \end{equation*}
    and the entropy kernel $\chi$ and entropy flux kernel $\sigma$ associated to the problem via \eqref{eq:ent kernel}--\eqref{eq:ent flux kernel} satisfy the kinetic equation 
    \begin{equation}\label{eq:kinetic def}
        \partial_t \chi (\rho(t,x),u(t,x),s) + \partial_x \sigma (\rho(t,x),u(t,x),s) = \partial^2_s \mu(t,x,s),
    \end{equation}
    in the sense of distributions on $\mathbb{R}^2_+\times\mathbb{R}$.
\end{enumerate}
\end{defn}

Our main result is the following.

\begin{thm}\label{thm:main ch2}
Let $(\rho_0,u_0)$ be locally integrable initial data of relative finite-energy, that is, $(\rho_0,u_0) \in L^1_{loc}(\mathbb{R}^2_+)$ and $E[\rho_0,u_0] \leq E_0 < \infty$, with $\rho_0 \geq 0$ almost everywhere. Suppose additionally that the pressure $p(\rho)$ satisfies the criteria for an asymptotically isothermal gas in the  sense of Definition \ref{def:asymp isothermal}. Then there exists a sequence of regularised initial data $(\rho_0^\varepsilon,u_0^\varepsilon)$ such that the unique, smooth solutions $(\rho^\varepsilon,u^\varepsilon)$ of \eqref{eq:ns} with this initial data converge almost everywhere as $\varepsilon \to 0$, $(\rho^\varepsilon,\rho^\varepsilon u^\varepsilon) \to (\rho,\rho u)$, to a relative finite-energy entropy solution of the Euler equations \eqref{eq:euler} with initial data $(\rho_0,\rho_0 u_0)$, in the  sense of Definition \ref{def:entropy sol}.
\end{thm}

\begin{remark}
In the proof of this theorem, we will obtain convergence of $(\rho^\eps,\rho^\eps u^\eps)$ in  measure. In conjunction with the uniform estimates of Proposition \ref{lemma:unif estimates ch2}, we will therefore also obtain $(\rho^\varepsilon,\rho^\varepsilon u^\varepsilon) \to (\rho,\rho u)$ in $L^p_{loc}(\mathbb{R}^2_+)\times L^q_{loc}(\mathbb{R}^2_+)$ strongly for $p \in [1,2)$ and $q \in [1,3/2)$ and weakly for $p=2$, $q=3/2$.
\end{remark}

The proof of this theorem relies on the techniques of compensated compactness, initiated in this context by DiPerna \cite{DiPerna2}, and the finite-energy framework begun by LeFloch--Westdickenberg \cite{LeFlochWestdickenberg} and extended by the authors in \cite{SS1}. Our approach is the following. The global existence of a unique entropy kernel $\chi$ solving \eqref{eq:ent kernel} is guaranteed by \cite[Theorem 2.1]{ChenLeFloch}, however without estimates as $\rho\to\infty$. To obtain improved estimates in this regime, we seek a good characterisation of the kernel, and so generate the solution of the linear wave equation \eqref{eq:ent kernel} in two steps. First, in order to deal with the singularity in $k'(\rho)^2$ near the vacuum, we solve \eqref{eq:ent kernel} in the interval $(0, \rho_*]$, for some $\rho_*$ large to be chosen later, using the results of \cite[Theorem 2.1]{ChenLeFloch}. Then, as a second step, we produce new estimates for $\chi$ in the high density regime by solving again from $\rho_*$:
\begin{equation}\label{eq:far field}
    \left\lbrace \begin{aligned}
    & \chi_{\rho \rho} - k'(\rho)^2 \chi_{u u} = 0, \qquad &&\text{ for } (\rho,u) \in (\rho_*,\infty)\times\mathbb{R}, \\
    & \chi(\rho_*,u) = \chi(\rho_*,u), &&\text{ for }u\in\R, \\
    & \chi_\rho(\rho_*,u) = \chi_\rho(\rho_*,u), &&\text{ for } u\in\R.
    \end{aligned}\right.
\end{equation}
In order to do this, we split the entropy kernel into two distinct quantities: a re-scaled version of the kernel obtained in \cite{SS1}, which we call $\chi^{*}$, and a perturbation, called $\tilde{\chi}$. To this end, we make the following definitions.

\begin{defn}\label{def:g sharp and g flat}
We define kernels $g^\sharp(\rho,u-s)$ and $g^\flat(\rho,u-s)$  to be the solutions of
\begin{equation}\label{eq:g kernels}
    \begin{Bmatrix}
&g^{\sharp}_{\rho\rho}-\frac{\kappa_2}{\rho^2}g^{\sharp}_{uu} = 0,   &&g^{\flat}_{\rho\rho}-\frac{\kappa_2}{\rho^2}g^{\flat}_{uu} = 0,\\
&g^{\sharp}\rvert_{\rho=\rho_*} = 0, &&g^{\flat}\rvert_{\rho=\rho_*} = \delta_{u=0},\\
&g^{\sharp}_{\rho}\rvert_{\rho = \rho_*} = \delta_{u=0}, &&g^{\flat}_{\rho}\rvert_{\rho = \rho_*} = 0.
\end{Bmatrix}
\end{equation}
Recalling the kernels $\chi^\sharp$ and $\chi^\flat$, introduced in \cite[Theorem 1.3]{SS1}, we deduce that for $(\rho,u)\in[\rho_*,\infty)\times\mathbb{R}$, $g^\sharp$ and $g^\flat$ satisfy the explicit formulas
\begin{equation}
    g^\sharp(\rho,u) := \rho_* \chi^\sharp\left( \frac{\rho}{\rho_*},\frac{u}{\sqrt{\kappa_2}} \right) \quad \text{and} \quad g^\flat(\rho,u) := \chi^\flat \left( \frac{\rho}{\rho_*},\frac{u}{\sqrt{\kappa_2}} \right).
\end{equation}
\end{defn}

We are now in a position to define $\chi^{*}$, the \emph{re-scaled approximately isothermal kernel}.
\begin{defn}\label{def:chi iso def}
We define, for $(\rho,u)\in[\rho_*,\infty)\times\mathbb{R}$,
\begin{equation}
    \chi^{*}(\rho,u) := \int_\mathbb{R} \chi_\rho(\rho_*,s) g^\sharp(\rho,u-s) \, ds + \int_\mathbb{R} \chi(\rho_*,s) g^\flat(\rho,u-s) \, ds.
\end{equation}
\end{defn}
We deduce from \eqref{eq:g kernels} and from \cite[Theorem 2.1]{ChenLeFloch} that $\chi^{*}$ is the unique solution of
\begin{equation}
    \left\lbrace\begin{aligned}
    & \chi^{{*}}_{\rho\rho} - \frac{\kappa_2}{\rho^2}\chi^{*}_{uu} = 0, \quad &&\text{for } (\rho,u) \in (\rho_*,\infty)\times\mathbb{R}, \\
    &\chi^{*}(\rho_*,u) = \chi(\rho_*,u), &&\text{for } u\in\R,\\
    &\chi^{*}_\rho(\rho_*,u) = \chi_\rho(\rho_*,u), &&\text{for } u\in\R.
    \end{aligned}\right.
\end{equation}
\begin{defn}
We define the \emph{perturbation kernel} $\tilde{\chi}$ as the difference between $\chi$ and $\chi^{*}$:
\begin{equation}\label{eq:chi error def}
    \tilde{\chi}(\rho,u) := \chi(\rho,u) - \chi^{*}(\rho,u) \qquad \text{for } (\rho,u)\in[\rho_*,\infty)\times\mathbb{R}.
\end{equation}
\end{defn}

Direct computation shows that $\tilde{\chi}$ satisfies the following linear wave equation,
\begin{equation}\label{eq:chi error wave eqn}
    \left\lbrace\begin{aligned}
    & \tilde{\chi}_{\rho \rho} - k'(\rho)^2\tilde{\chi}_{u u} = \big( k'(\rho)^2 - \frac{\kappa_2}{\rho^2} \big)\chi^{*}_{uu},  \quad &&\text{for } (\rho,u) \in (\rho_*,\infty)\times\mathbb{R}, \\
    & \tilde{\chi}(\rho_*,u) = 0, &&\text{for } u\in\R,\\
    & \tilde{\chi}_\rho (\rho_*,u) = 0, &&\text{for } u\in\R.
    \end{aligned}\right.
\end{equation}
The key idea of this paper is to represent $\tilde\chi(\rho,u)$ using a representation formula through integrals (of $\tilde\chi$) along backwards characteristics (see Lemma \ref{lemma:chi rep form} and \eqref{eq:rep form phi} below). Such a representation formula allows us to make precise estimates on the growth of $\tilde\chi(\rho,u)$ and its derivatives as $\rho\to\infty$. Using the estimates of \cite{SS1} for the principal part, $\chi^*(\rho,u)$, we hence also obtain estimates on $\chi(\rho,u)$. Moreover, the representation formula technique turns out to be very well adapted to giving estimates as $\rho\to\infty$ to many of the quantities associated to entropies and entropy fluxes and will be used extensively throughout this paper. All of the uniform estimates required to establish the compactness of the entropy dissipation measures and to reduce the support of the Young measure generated by the viscous approximations $(\rho^\varepsilon, u^\varepsilon)$ are then verified. These observations are enough to prove Theorem \ref{thm:main ch2} rigorously.

The paper is structured as follows. In Section \ref{section:elementary quantities}, we introduce some elementary quantities linked to the pressure, and compute estimates on them that are essential throughout. Some of the proofs of these estimates are contained in Appendix \ref{section:appendix}, which also contains auxiliary results on the internal energy function $e(\rho)$. In Section \ref{section:rep formula chi error}, we obtain a representation formula for $\tilde{\chi}$, which enables us to obtain a uniform $L^\infty$ estimate via a Gr\"{o}nwall type argument. With this, in Section \ref{section:entropy pairs bounds}, we estimate a special entropy $\hat{\eta}$ and its flux $\hat{q}$ (see Lemma \ref{lemma:pointwise hat chapter 2}) which are necessary to obtain higher uniform integrability estimates on the viscous approximations, and we also bound entropies generated by compactly supported test functions (see Lemma \ref{lemma:cpctly supported chapter 2}). Similar procedures give rise to estimates on the derivatives of these entropies, which we also outline in detail in Section \ref{section:entropy pairs bounds}. In Section \ref{section:uniformestimates}, we collect together the necessary uniform estimates on $(\rho^\eps,u^\eps)$ and the compactness of the entropy dissipation measures. In Section \ref{section:singularities ch2}, we calculate the structure of the singularities of the entropy kernel, and prove a result akin to Lemma 2.7 of \cite{SS1}, which was used crucially in \cite[Section 5]{SS1}. Having established this, we are able to prove the reduction of support of the Young measure generated by the approximate solutions using compensated compactness techniques and hence deduce the strong convergence of the viscous solutions. This is contained in Section \ref{section:reduction ch2}, which ends with a proof of the main result.

\section{Elementary quantities}\label{section:elementary quantities}
 In this section, we define elementary quantities related to the pressure, and make note of some of their properties which will be essential in the proof of the main result. We fix $\rho_*>0$ here, to be determined later. To begin with, by analogy with the definition of $k(\rho)$ given in \eqref{def:k(rho)}, we make the following definition.
\begin{defn}
We define the quantity
\begin{equation}\label{eq:k star def}
    k_*(\rho) := \int_{\rho_*}^\rho \frac{\sqrt{\kappa_2}}{y} \, dy + k(\rho_*) = \sqrt{\kappa_2}\log(\rho/\rho_*) + k(\rho_*) \qquad \text{for all } \rho \geq \rho_*.
\end{equation}
\end{defn}
Note that $k_*'(\rho)$ is the speed of propagation for the entropy equation of a gas with an isothermal pressure law $p(\rho) = \kappa_2 \rho$. Meanwhile, $k'(\rho)$ is the speed of propagation for the actual pressure law $p(\rho)$, that is, the asymptotically isothermal gas, as can be seen directly from \eqref{eq:ent kernel}.
\begin{remark}
Observe that $k_*'(\rho) = \frac{\sqrt{\kappa_2}}{\rho}$ and $k(\rho_*) = k_*(\rho_*)$. Thus,
\begin{equation}
    k(\rho) - k_*(\rho) = \int_{\rho_*}^\rho \frac{\sqrt{p'(y)} - \sqrt{\kappa_2}}{y} \, dy \qquad \text{for all } \rho \geq \rho_*.
\end{equation}
Since $\kappa_2 = \lim_{\rho \to \infty} p'(\rho)$, and $p \in C^4((0,\infty))$, we can rewrite the above as
\begin{equation}\label{eq:k and k star diff}
    k(\rho) - k_*(\rho) = -\int_{\rho_*}^\rho \frac{1}{y} \left( \int_{y}^\infty \frac{p''(z)}{2\sqrt{p'(z)}} \, dz \right) \, dy.
\end{equation}
\end{remark}

\begin{defn}\label{def:d and d'}
We define the quantities $d(\rho)$ and $d_*(\rho)$ by
\begin{equation}
    d(\rho) := 2 + (\rho-\rho_*)\frac{k''(\rho)}{k'(\rho)}, \qquad d_*(\rho) := 2 + (\rho-\rho_*)\frac{k''_*(\rho)}{k'_*(\rho)} \qquad \text{for } \rho \geq \rho_*.
\end{equation}
Note that both of these quantities are strictly positive on the interval $[\rho_*,\infty)$, due to  assumption \eqref{eq:st hyp gen non}.
\end{defn}

Quantities $d(\rho)$ and $d_*(\rho)$ will appear in the representation formulas below. The following key lemma will play an important role throughout Sections \ref{section:rep formula chi error} and \ref{section:entropy pairs bounds}.

\begin{lemma}\label{lemma:p' below}
For all $\rho \geq R$, we have
\begin{equation}
    |p'(\rho) - \kappa_2| \leq 2C_p \rho^{-\alpha}, \qquad  |p^{(j)}(\rho)| \leq (j+1)C_p \rho^{-\alpha-(j-1)} \qquad \text{for } j = 2,3,4.
\end{equation}
As such, choosing $\rho_* \geq \max\{ R, (4C_p/\kappa_2)^{1/\alpha}\}$,
\begin{equation}\label{eq:p' below bound}
    \rho^2 \leq \frac{4\kappa_2}{3}\rho p(\rho), \quad\frac{\kappa_2}{2} \leq p'(\rho) \leq \frac{3 \kappa_2}{2}, \quad \sqrt{\frac{\kappa_2}{2}} \rho^{-1} \leq k'(\rho) \leq \sqrt{\frac{3\kappa_2}{2}} \rho^{-1} \quad \text{for  } \rho\geq \rho_*,
\end{equation}
and
\begin{equation}\label{eq:k and k star close}
|k(\rho) - k_*(\rho)| \leq \frac{3C_p}{\alpha^2\sqrt{2\kappa_2}}\rho_*^{-\alpha} \qquad \text{for  } \rho\geq \rho_*.
\end{equation}
\end{lemma}

\begin{proof}
Observe that, for $j \geq 1$,
\begin{equation*}
    \big(p(\rho) - \rho\kappa_2  \big)^{(j)}(\rho) = \rho \left( \frac{p(\rho)}{\rho} - \kappa_2 \right)^{(j)} + j \left( \frac{p(\rho)}{\rho}-\kappa_2 \right)^{(j-1)}.
\end{equation*}
Thus, using the bounds provided by \eqref{eq:p derivs}, we obtain the result. Also, from \eqref{eq:k and k star diff},
\begin{equation*}
    |k(\rho) - k_*(\rho)| \leq \frac{3C_p}{\alpha\sqrt{2\kappa_2}}\int_{\rho_*}^\rho y^{-\alpha-1} \, dy \leq \frac{3C_p}{\alpha^2\sqrt{2\kappa_2}}\rho_*^{-\alpha}.
\end{equation*}
\end{proof}

\begin{cor}\label{corollary:k' and k'' compare with rho powers}
Assume that $\rho_* \geq \max\{ R, (4C_p/\kappa_2)^{1/\alpha} \}$. Then there exists $M=M(\alpha,\kappa_2,C_p)$ such that
\begin{equation}\label{eq:k'' close to rho squared}
    \left| k'(\rho)^2-\frac{\kappa_2}{\rho^2} \right| + \left|- \frac{\sqrt{\kappa_2}}{\rho^2} - k''(\rho)\right| \leq M \rho^{-\alpha-2} \qquad \text{for } \rho \geq \rho_*.
\end{equation}
Meanwhile,
\begin{equation}\label{eq:k''' close to rho cubed}
    \left| \frac{2\sqrt{\kappa_2}}{\rho^3} - k^{(3)}(\rho) \right| + \rho \left| -\frac{6\sqrt{\kappa_2}}{\rho^4} - k^{(4)}(\rho) \right| \leq M \rho^{-\alpha-3} \qquad \text{for } \rho \geq \rho_*.
\end{equation}
It follows that there exists a positive constant $M=M(\alpha,\kappa_2,C_p,\rho_*)$ such that
\begin{equation}\label{eq:k'' and higher, decay}
    |k''(\rho)| + \rho|k^{(3)}(\rho)| + \rho^2|k^{(4)}(\rho)| \leq M\rho^{-2} \qquad \text{for } \rho \geq \rho_*.
\end{equation}
\end{cor}
The proof is omitted to Appendix \ref{section:appendix}, which also contains a variety of other useful estimates on these quantities, along with $d(\rho)$ and $d_*(\rho)$.

\section{Entropy and Entropy Flux Kernels}\label{section:rep formula chi error}
We begin this section by recalling the following theorem on the existence and regularity of the entropy and entropy flux kernels, $\chi(\rho,u-s)$ and $\sigma(\rho,u,s)$ due to Chen and LeFloch \cite{ChenLeFloch}.
\begin{thm}[{\cite[Theorems 2.1--2.3]{ChenLeFloch}}]\label{thm:expansion at vacuum}
Assume that the pressure $p \in C^1([0,\infty))\cap C^4((0,\infty))$ satisfies assumptions \eqref{eq:st hyp gen non} and \eqref{eq:p asymp near 0 P}. Then \eqref{eq:ent kernel} and \eqref{eq:ent flux kernel} admit global unique H\"{o}lder continuous solutions 
$\chi(\rho,u,s) = \chi(\rho,u-s)$ and $\sigma(\rho,u,s)=u\chi(\rho,u-s)+h(\rho,u-s)$. 
Additionally,
\begin{equation}\label{eq:kernel support intro}
\supp \chi(\rho,\cdot) , \supp h(\rho,\cdot) = [-k(\rho),k(\rho)] \qquad \text{for } \rho \geq 0,
\end{equation}
and, for each $\rho >0$, $\chi(\rho,\cdot)$ is strictly positive on $(-k(\rho),k(\rho))$. 
Moreover, for any fixed $\rho_*>0$, the entropy kernel satisfies, for $0<\rho\leq\rho_*$,
\beq
\|\chi(\rho,\cdot)\|_{L^\infty(\R)}\leq C(\rho_*)\rho^{1-\th},\qquad [\chi(\rho,\cdot)]_{C^{\tilde\la}(\R)}\leq C(\rho_*),
\eeq
where $[\cdot]_{C^{\tilde{\lambda}}(\mathbb{R})}$ is the H\"{o}lder seminorm, $\lambda = \frac{3-\gamma}{2(\gamma-1)}$, $\tilde\la=\min(\la,1)$, and $C(\rho_*)$ is a positive constant.
\end{thm}
We recall that the exponent $\la$ is related to $\th$ by $2\la\th=1-\th$.

In this section, we derive two representation formulas for the kernel $\tilde{\chi}$. We obtain the first representation by taking the difference of the representation formulas for $\chi$ and $\chi^{*}$ (see the proof of Lemma \ref{lemma:chi rep form}), and the second by directly considering the linear wave equation for $\tilde{\chi}$, namely \eqref{eq:chi error wave eqn}. The first representation formula is required to estimate $\Vert \tilde{\chi}(\rho,\cdot)\Vert_{L^\infty(\mathbb{R})}$ in Lemma \ref{lemma:error bound abs} below. Armed with this bound, we are able to compute the derivatives with respect to $u$ of the entropies generated by the perturbation, which is required to show that the second representation is well-defined in Lemma \ref{lemma:second rep formula eta error}.

At this point, we fix $\rho_* \geq \max\lbrace R, (4C_p/\kappa_2)^{1/\alpha} \rbrace$, so that all of the estimates of Section \ref{section:elementary quantities} hold. We also fix $\alpha \in (0,1)$ for clarity of exposition, though we note that Theorem \ref{thm:main ch2} holds true for any $\alpha > 0$ and that better estimates are available for $\alpha \geq 1$ (\textit{cf.}~\cite[Chapter 3]{thesis}). We also make note of the following lemma, which outlines some important properties of $\chi^{*}$. These follow directly from the analysis of the entropy kernel considered in \cite{SS1}, and will be indispensable for our later estimates.

\begin{prop}\label{lemma:chi iso rescaled properties}
There exists a positive constant $M$ depending on $\rho_*$ such that, with $\tilde{\lambda} := \min \{ \lambda , 1 \}$,
\begin{equation}\label{eq:chi star holder norm}
    \Vert \chi^{*}(\rho,\cdot) \Vert_{L^\infty(\mathbb{R})} + [ \chi^{*}(\rho,\cdot) ]_{C^{\tilde{\lambda}}(\mathbb{R})} \leq M \frac{\rho}{\sqrt{k(\rho)}} \qquad \text{for all } \rho \geq \rho_*.
\end{equation}
\end{prop}

\begin{remark}\label{remark:supports}
By Theorem \ref{thm:expansion at vacuum}, 
\begin{equation}\label{eq:supp chi}
    \supp \chi(\rho,\cdot) = [-k(\rho),k(\rho)] =: \mathcal{K}.
\end{equation} Similarly, for $\rho \geq \rho_*$,
\begin{equation}\label{eq:supp chi iso}
    \supp \chi^{*}(\rho,\cdot) = [-k_*(\rho),k_*(\rho)] =: \mathcal{K}^{*}.
    \end{equation}
    We therefore deduce from \eqref{eq:chi error def} that
\begin{equation}\label{eq:supp chi error}
    \supp \tilde{\chi}(\rho,\cdot) \subset [-\max\{k(\rho),k_*(\rho)\},\max\{k(\rho),k_*(\rho)\}] =: \tilde{\mathcal{K}}.
\end{equation}
\end{remark}

\subsection{First representation formula and uniform estimate on the perturbation}

We begin with a representation formula for the entropy kernel, representing the kernel $\chi$ in terms of its integral along characteristic curves (note that the curves $u\pm k(\rho)=\text{const.}$ are characteristics), \textit{cf.}~\cite[Equation (3.38)]{ChenLeFloch}.
\begin{lemma}\label{lemma:chi rep form}
Given any $(\rho_0,u_0) \in \mathcal{K}$ and any $0 < \rho_* < \rho_0$, we have
\begin{equation}\label{eq:rep form chi}
\begin{aligned}
\chi(\rho_0,u_0) = &\, \frac{1}{2(\rho_0 - \rho_*)k'(\rho_0)}\int_{\rho_*}^{\rho_0} k'(\rho) d(\rho) \chi(\rho,u_0+k(\rho_0)-k(\rho)) \, d\rho \\
&+ \frac{1}{2(\rho_0 - \rho_*)k'(\rho_0)}\int_{\rho_*}^{\rho_0} k'(\rho) d(\rho) \chi(\rho,u_0-k(\rho_0)+k(\rho)) \, d\rho \\
& - \frac{1}{2(\rho_0 - \rho_*)k'(\rho_0)} \int_{-(k(\rho_0)-k(\rho_*))}^{k(\rho_0)-k(\rho_*)} \chi(\rho_*,u_0-y) \, dy.
\end{aligned}
\end{equation}
\end{lemma}
\begin{proof}
Fix any $(\rho_0,u_0) \in \mathcal{K}$  with $\rho_0>\rho_*$. Then from applying the Fourier transform with respect to $u$ to the equation \eqref{eq:ent kernel} for the entropy kernel $\chi$ we see that, for any $\rho\in(\rho_*,\rho_0)$,
\begin{equation*}
(\rho-\rho_*)k'(\rho)^2 \mathcal{F}\chi(\rho,\xi) = -(\rho-\rho_*)\xi^{-2} \mathcal{F}\chi_{\rho \rho}(\rho,\xi).
\end{equation*}
Multiplying this by $\sin(( k(\rho)-k(\rho_0))\xi)$ and integrating in the interval $\rho\in[\rho_*,\rho_0]$ we get
\begin{equation*}
\begin{aligned}
\int_{\rho_*}^{\rho_0} (\rho-\rho_*) k'(\rho)^2 \sin(( & k(\rho)-k(\rho_0))\xi)  \mathcal{F}\chi(\rho,\xi) \, d\rho\\
= & -\xi^{-2}\int_{\rho_*}^{\rho_0} (\rho-\rho_*) \sin ((k(\rho)-k(\rho_0))\xi) \mathcal{F}\chi_{\rho \rho}(\rho,\xi) \, d\rho, \\
= &\, \xi^{-2}\int_{\rho_*}^{\rho_0} \partial_{\rho} \big( (\rho-\rho_*) \sin((k(\rho)-k(\rho_0))\xi) \big) \mathcal{F}\chi_{\rho}(\rho,\xi) \, d\rho \\
& - \xi^{-2}[(\rho-\rho_*)\sin ((k(\rho)-k(\rho_0))\xi)\mathcal{F}\chi_{\rho}(\rho,\xi)]_{\rho_*}^{\rho_0}, \\
= & -\xi^{-2}\int_{\rho_*}^{\rho_0} \partial^2_{\rho} \big[ (\rho-\rho_*) \sin((k(\rho)-k(\rho_0))\xi) \big] \mathcal{F}\chi(\rho,\xi) \, d\rho \\
& + \xi^{-2} \bigg[ \partial_{\rho}\big[ (\rho-\rho_*) \sin ((k(\rho)-k(\rho_0))\xi) \big] \mathcal{F}\chi(\rho,\xi) \bigg]_{\rho_*}^{\rho_0},
\end{aligned}
\end{equation*}
and we find that the final line can be written as
\begin{equation*}
\begin{aligned}
-\xi^{-1} \int_{\rho_*}^{\rho_0} \Big( k' &+ \big( (\rho-\rho_*)k' \big)' \Big) \cos ((k(\rho)-k(\rho_0))\xi) \mathcal{F}\chi(\rho,\xi) \, d\rho \\
&+ \int_{\rho_*}^{\rho_0} (\rho-\rho_*) k'(\rho)^2 \sin((k(\rho)-k(\rho_0))\xi) \mathcal{F}\chi(\rho,\xi) \, d\rho \\
& + \xi^{-1} (\rho_0-\rho_*)k'(\rho_0) \mathcal{F}\chi(\rho_0,\xi) - \xi^{-2} \sin((k(\rho_*)-k(\rho_0))\xi)\mathcal{F}\chi(\rho_*,\xi).
\end{aligned}
\end{equation*}
Hence, we have
\begin{equation*}
\begin{aligned}
(\rho_0-\rho_*)k'(\rho_0)\mathcal{F}\chi & (\rho_0,\xi) = \\
& \xi^{-1} \sin ((k(\rho_*)-k(\rho_0))\xi) \mathcal{F}\chi(\rho_*,\xi) \\
& + \int_{\rho_*}^{\rho_0} \Big( k'+ \big( (\rho-\rho_*) k' \big)' \Big) \cos ((k(\rho)-k(\rho_0))\xi)\mathcal{F}\chi(\rho,\xi) \, d\rho. 
\end{aligned}
\end{equation*}
Now, observe that for $a > 0$ and $g \in \mathcal{S}'(\mathbb{R}) \cap C(\mathbb{R})$, we have
\begin{equation*}
\mathcal{F}^{-1} \big( \xi^{-1}\sin (a\xi) \mathcal{F}g(\xi) \big)(x) = \frac{1}{2}\int_{-a}^{a} g(x-y) \, dy,
\end{equation*}
and
\begin{equation*}
    \mathcal{F}^{-1} \left( \cos(a\xi)\mathcal{F}g(\xi) \right) (x) = \frac{1}{2}\left( g(x+a) + g(x-a) \right).
\end{equation*}
So, applying the inverse Fourier transform we obtain the result as claimed.
\end{proof}

\begin{remark}\label{remark:same rep form}
The same representation formula as \eqref{eq:rep form chi} holds for $\chi^{*}$, except $k$ should be replaced with $k_*$, $d$ with $d_*$, and $\mathcal{K}$ with $\mathcal{K}^{*}$.
\end{remark}
The most significant result of this section is the following estimate on the growth rate of $\chi(\rho,u-s)$ as $\rho\to\infty$.
\begin{lemma}\label{lemma:error bound abs}
There exists a constant $M>0$ such that
\begin{equation}
\Vert \tilde{\chi}(\rho,\cdot) \Vert_{L^\infty(\mathbb{R})} \leq M\rho^{1-\alpha \tilde{\lambda}} \qquad \text{for } \rho \geq \rho_*,
\end{equation}
where we recall that $\tilde{\lambda} = \min \{ \lambda , 1 \}$.
\end{lemma}

\begin{proof}
In view of Lemma \ref{lemma:chi rep form} and Remark \ref{remark:same rep form}, by subtracting $\chi^{*}(\rho_0,u_0)$ from $\chi(\rho_0,u_0)$ and recalling that $\chi(\rho_*,\cdot) = \chi^{*}(\rho_*,\cdot)$, we arrive at the first representation formula for the perturbation. Given any $(\rho_0,u_0) \in \tilde{\mathcal{K}}$, 
\begin{equation}\label{eq:rep form phi}
    \begin{aligned}
          2k'(\rho_0)(\rho_0 & -\rho_*) \tilde{\chi} (\rho_0,u_0)  \\ 
        =\,\,\,&  \int_{\rho_*}^{\rho_0} {k'(\rho)}d(\rho) \Big[ \tilde{\chi} (\rho,u_0+k(\rho_0)-k(\rho)) + \tilde{\chi}(\rho,u_0-k(\rho_0)+k(\rho)) \Big] \, d\rho \\ 
        + &  \int_{\rho_*}^{\rho_0} {k'(\rho) d(\rho)} \Big[ \chi^{*}(\rho,u_0+k(\rho_0)-k(\rho))  - \chi^{*}(\rho,u_0+k_*(\rho_0)-k_*(\rho))\Big] \, d\rho \\ 
        + & \int_{\rho_*}^{\rho_0} \Big[ {k'(\rho) d(\rho)} - {k_*'(\rho)d_*(\rho)}\frac{k'(\rho_0)}{k_*'(\rho_0)} \Big] \chi^{*}(\rho,u_0+k_*(\rho_0)-k_*(\rho)) \, d\rho \\
        + & \int_{\rho_*}^{\rho_0} {k'(\rho) d(\rho)} \Big[ \chi^{*}(\rho,u_0-k(\rho_0)+k(\rho)) - \chi^{*}(\rho,u_0 - k_*(\rho_0) + k_*(\rho)) \Big] \, d\rho \\
        + &  \int_{\rho_*}^{\rho_0} \Big[ {k'(\rho)d(\rho)} - {k'_*(\rho) d_*(\rho)}\frac{k'(\rho_0)}{k'_*(\rho_0)} \Big] \chi^{*}(\rho,u_0 - k_*(\rho_0) + k_*(\rho)) \, d\rho \\
        - &  \bigg[ \int_{-(k(\rho_0) - k(\rho_*))}^{k(\rho_0) - k(\rho_*)} {\chi^{*}(\rho_*,u_0 - y)}\, dy - \int_{-(k_*(\rho_0) - k_*(\rho_*))}^{k_*(\rho_0) - k_*(\rho_*)} \frac{k'(\rho_0)}{k_*'(\rho_0)}{\chi^{*}(\rho_*,u_0 - y)} \, dy \bigg]\\
        =\,\,\,&\,I_1+\cdots+I_6.
    \end{aligned}
\end{equation}
 We then bound $|I_1|$ by
\begin{equation*}
   |I_1|\leq \int_{\rho_*}^{\rho_0} 2 d(\rho) \Vert k'(\rho) \tilde{\chi}(\rho,\cdot) \Vert_{L^\infty(\mathbb{R})} \, d\rho.
\end{equation*}
$I_2$ is bounded by
\beqas
    |I_2|\leq&\, \int_{\rho_*}^{\rho_0} d(\rho) [ k'(\rho) \chi^{*}(\rho,\cdot) ]_{C^{\tilde{\lambda}}(\mathbb{R})} \big| (k(\rho_0)-k(\rho)) - (k_*(\rho_0)-k_*(\rho)) \big|^{\tilde{\lambda}} \, d\rho\\
    \leq&\,  M \int_{\rho_*}^{\rho_0} \rho^{-\alpha \tilde{\lambda}}  [ k'(\rho) \chi^{*}(\rho,\cdot) ]_{C^{\tilde{\lambda}}(\mathbb{R})} \, d\rho,
\eeqas
where we have applied Lemma \ref{lemma:k difference rho0 and rho} and the bound \eqref{eq:bound on d}, and
$M = M(\alpha,\kappa_2,C_p)$. Notice that $I_4$ can be bounded in exactly the same way. Next, $I_3$ is bounded by
\begin{equation*}
|I_3|\leq     \int_{\rho_*}^\rho \left| d(\rho) - d_*(\rho) \frac{k'(\rho_0)}{k_*'(\rho_0)} \frac{k_*'(\rho)}{k'(\rho)} \right| \Vert k'(\rho) \chi^{*}(\rho,\cdot)  \Vert_{L^\infty(\mathbb{R})} \, d\rho,
\end{equation*}
and we bound the first term in the integrand by
\begin{equation*}
    d(\rho) \left| 1-\frac{k'(\rho_0)}{k_*'(\rho_0)} \frac{k_*'(\rho)}{k'(\rho)} \right| + \left| d(\rho) - d_*(\rho) \right|\frac{k'(\rho_0)}{k_*'(\rho_0)} \frac{k_*'(\rho)}{k'(\rho)} \leq M\rho^{-\al}
\end{equation*}
using Lemma \ref{lemma:d and d star} for the right-hand term and Lemma \ref{lemma:k' product bound} for the left-hand term, respectively, along with \eqref{eq:bound on d}. Thus we see that 
\begin{equation*}
  |I_3|\leq  M \int_{\rho_*}^{\rho_0} \rho^{-\alpha}  \Vert k'(\rho) \chi^{*}(\rho,\cdot) \Vert_{L^\infty(\mathbb{R})} \, d\rho,
\end{equation*}
where $M = M(\alpha,\kappa_2,C_p)$. Once again, $I_5$ may be bounded in exactly the same way. We split $I_6$ into two terms as
\begin{equation}\label{eq:split rep form chi}
\begin{aligned}
 I_6=&   \int_{-(k(\rho_0)-k(\rho_*))}^{k(\rho_0)-k(\rho_*)} \chi^{*}(\rho_*,u_0-y) \, dy - \int_{-(k_*(\rho_0)-k_*(\rho_*))}^{k_*(\rho_0)-k_*(\rho_*)} \chi^{*}(\rho_*,u_0-y) \, dy \\
    &+ \left( 1-\frac{k'(\rho_0)}{k_*'(\rho_0)} \right)  \int_{-(k_*(\rho_0)-k_*(\rho_*))}^{k_*(\rho_0)-k_*(\rho_*)} \chi^{*}(\rho_*,u_0-y) \, dy.
\end{aligned}
\end{equation}
We concentrate on the first line of \eqref{eq:split rep form chi}. Observe that the two intervals, $[-(k(\rho_0)-k(\rho_*)),k(\rho_0)-k(\rho_*)]$ and $[-(k_*(\rho_0)-k_*(\rho_*)),k_*(\rho_0)-k_*(\rho_*)]$, are always nested within one another; one interval is always entirely contained in the other. Hence, since the same quantity is being integrated, we may bound this line by
\begin{equation*}
    2\big| (k(\rho_0)-k(\rho_*)) - (k_*(\rho_0)-k_*(\rho_*)) \big| \cdot \Vert \chi^{*}(\rho_*,\cdot) \Vert_{L^\infty(\mathbb{R})}\leq M\rho_*^{-\al}\min\big(\frac{\rho_0-\rho_*}{\rho_*},1\big)\|\chi^{*}(\rho_*,\cdot) \Vert_{L^\infty(\mathbb{R})},
\end{equation*}
by Lemma \ref{lemma:k difference rho0 and rho}.  On the other hand, the final term of \eqref{eq:split rep form chi} is bounded by
\begin{equation*}
    \left| 1-\frac{k'(\rho_0)}{k_*'(\rho_0)} \right| \cdot 2\big| k_*(\rho_0) - k_*(\rho_*) \big| \cdot \Vert \chi^{*}(\rho_*,\cdot) \Vert_{L^\infty(\mathbb{R})}\leq M \rho_0^{-\alpha} \log \left( \frac{\rho_0}{\rho_*} \right) \Vert \chi^{*}(\rho_*,\cdot) \Vert_{L^\infty(\mathbb{R})},
\end{equation*}
by Lemma \ref{lemma:k' k'star ratio}. We emphasise that it has been sufficient to assume $\alpha>0$.

In total, by combining these estimates for $I_1,\dots,I_6$ and then dividing through \eqref{eq:rep form phi} by the factor $(\rho_0-\rho_*)$, we obtain
\begin{equation}\label{eq:pre gronwall}
\begin{aligned}
     \Vert k'(\rho_0) \tilde{\chi}(\rho_0,\cdot) \Vert_{L^\infty(\mathbb{R})} \leq &\, \frac{1}{(\rho_0-\rho_*)}\int_{\rho_*}^{\rho_0} d(\rho) \Vert k'(\rho) \tilde{\chi}(\rho,\cdot) \Vert_{L^\infty(\mathbb{R})} \, d\rho \\
     &+ \frac{M}{(\rho_0-\rho_*)}\int^{\rho_0}_{\rho_*} \rho^{-\alpha \tilde{\lambda}} [ k'(\rho) \chi^{*}(\rho,\cdot) ]_{C^{\tilde{\lambda}}(\mathbb{R})} \, d\rho \\
     &+ \frac{M}{(\rho_0-\rho_*)} \int_{\rho_*}^{\rho_0} \rho^{-\alpha} \Vert k'(\rho) \chi^{*}(\rho,\cdot) \Vert_{L^\infty(\mathbb{R})} \, d\rho \\
     &+ \frac{M}{1+(\rho_0-\rho_*)}\rho_*^{-\alpha} \Vert \chi^{*}(\rho_*,\cdot) \Vert_{L^\infty(\mathbb{R})} \\
     &+ \frac{M}{1+(\rho_0-\rho_*)} \rho_0^{-\alpha} \log \left( 1 +  \frac{\rho_0}{\rho_*} \right) \Vert \chi^{*}(\rho_*,\cdot) \Vert_{L^\infty(\mathbb{R})}.
\end{aligned}
\end{equation}
We now apply Gr\"{o}nwall's lemma to \eqref{eq:pre gronwall} and divide by $k'(\rho_0)$, which yields
\begin{equation}\label{eq:post gronwall}
\begin{aligned}
     \Vert  \tilde{\chi}&  (\rho_0,\cdot)   \Vert_{L^\infty(\mathbb{R})}  \\ \leq &
      \bigg(  \frac{M}{(\rho_0-\rho_*)k'(\rho_0)}\int^{\rho_0}_{\rho_*} \rho^{-\alpha \tilde{\lambda}} [ k'(\rho) \chi^{*}(\rho,\cdot) ]_{C^{\tilde{\lambda}} (\mathbb{R})} \, d\rho \\
     &+ \frac{M}{(\rho_0-\rho_*)k'(\rho_0)} \int_{\rho_*}^{\rho_0} \rho^{-\alpha}  \Vert k'(\rho) \chi^{*}(\rho,\cdot) \Vert_{L^\infty(\mathbb{R})} \, d\rho \\
     &+ \frac{M}{k'(\rho_0)[1+(\rho_0-\rho_*)]}\rho_*^{-\alpha} \Vert \chi^{*}(\rho_*,\cdot) \Vert_{L^\infty(\mathbb{R})} \\
     &+ \frac{M}{k'(\rho_0)[1+(\rho_0-\rho_*)]} \rho_0^{-\alpha} \log \left( 1 + \frac{\rho_0}{\rho_*} \right) \Vert \chi^{*}(\rho_*,\cdot) \Vert_{L^\infty(\mathbb{R})} \bigg) \exp\Big({\frac{1}{(\rho_0-\rho_*)}\int_{\rho_*}^{\rho_0} d(\rho) \, d\rho }\Big).
\end{aligned}
\end{equation}
It is clear that, in view of the bound on $d(\rho)$ provided by Lemma \ref{lemma:d and d star}, the exponential on the right-hand side is bounded above by $e^3$. Now recall from Proposition \ref{lemma:chi iso rescaled properties} that
\begin{equation*}
    \Vert \chi^{*}(\rho,\cdot) \Vert_{L^\infty(\mathbb{R})} + [ \chi^{*}(\rho,\cdot) ]_{C^{\tilde{\lambda}}(\mathbb{R})} \leq M\rho \qquad \text{for all } \rho \geq \rho_*.
\end{equation*}
Then, using also Lemma \ref{lemma:k' k'star ratio}, we have
\begin{equation*}
    \Vert k'(\rho) \chi^{*}(\rho,\cdot)\Vert_{L^\infty(\mathbb{R})} + [k'(\rho) \chi^{*}(\rho,\cdot)]_{C^{\tilde{\lambda}}(\mathbb{R})} \leq 
    M \qquad \text{for all }\rho \geq \rho_*.
\end{equation*}
We are now in a position to bound each of the terms in the  brackets of \eqref{eq:post gronwall}. We begin with the first term.
\begin{equation}\label{eq:pre pieces1}
    \begin{aligned}
         \frac{M}{(\rho_0-\rho_*)k'(\rho_0)} \int_{\rho_*}^{\rho_0} \rho^{-\alpha \tilde{\lambda}}  [k'(\rho) \chi^{*}(\rho,\cdot)]_{C^{\tilde{\lambda}}(\mathbb{R})} \, d\rho \leq \frac{M}{(\rho_0-\rho_*)k'(\rho_0)} \int_{\rho_*}^{\rho_0} \rho^{-\alpha \tilde{\lambda}} \, d\rho\leq M\rho_0^{1-\al\tilde\la},
    \end{aligned}
\end{equation}
since $\alpha \in (0,1)$ implies $\alpha \tilde{\lambda} \in (0,1)$.

Similarly, for the second term of \eqref{eq:post gronwall}, noting again that $\alpha \in (0,1)$, 
\begin{equation}\label{eq:pre pieces2}
    \frac{M}{(\rho_0 - \rho_*) k'(\rho_0)} \int_{\rho_*}^{\rho_0} \rho^{-\alpha} \Vert k'(\rho) \chi^{*}(\rho,\cdot) \Vert_{L^\infty(\mathbb{R})} \, d\rho \leq \frac{M}{(\rho_0-\rho_*)k'(\rho_0)} \int_{\rho_*}^{\rho_0} \rho^{-\alpha} \, d\rho \leq M\rho_0^{1-\alpha}.
\end{equation}

For the third and fourth terms of \eqref{eq:post gronwall}, observe that, by Lemma \ref{lemma:p' below},
\begin{equation}\label{eq:piece3}
    \frac{1}{k'(\rho_0)[1+(\rho_0-\rho_*)]} \leq M.
\end{equation}
Thus, collecting the results in \eqref{eq:pre pieces1}--\eqref{eq:piece3}, we have that, for some positive $M=M(\alpha,\kappa_2,C_p,\rho_*)$, 
\begin{equation*}
    \Vert \tilde{\chi} (\rho_0 , \cdot) \Vert_{L^\infty(\mathbb{R})} \leq M\left( \rho_0^{1-\alpha \tilde{\lambda}} + \rho_0^{1-\alpha} +1 \right),
\end{equation*}
which (by making $M$ larger if necessary) straightforwardly yields the result.
\end{proof}
\begin{remark}
We note that, if $\alpha \tilde{\lambda} > 1$, the proof of Lemma \ref{lemma:error bound abs} shows that we have the stronger estimate
\begin{equation*}
    \Vert \tilde{\chi}(\rho,\cdot) \Vert_{L^\infty(\mathbb{R})} \leq M \qquad \text{for } \rho \geq \rho_*.
\end{equation*}
\end{remark}

\subsection{Generating entropies and second representation formula}\label{subsection:generating entropies}
In what follows, we generate entropies by convolving with the entropy kernel. Specifically, given a suitably integrable test function $\psi$, we may generate an entropy according to the formula
\begin{equation}\label{eq:generate any psi entropy}
    \eta^\psi(\rho,u) := \int_\mathbb{R} \chi(\rho,u-s) \psi(s) \, ds.  
    \end{equation}
    Thus, for $\rho\geq \rho_*$,
    \begin{equation}
    \eta^\psi(\rho,u)= \int_\mathbb{R} \chi^{*}(\rho,u-s) \psi(s) \, ds + \int_\mathbb{R} \tilde{\chi}(\rho,u-s) \psi(s) \, ds  =: \eta^{*,\psi}(\rho,u) + \tilde{\eta}^{\psi}(\rho,u).
\end{equation}
\begin{remark}
 Since $\chi$ is given as a function of $(\rho,u)$, we will also generate entropies as functions of $(\rho,u)$, even though these are technically functions of $(\rho,m)$. We adopt this convention throughout this subsection and Section \ref{section:entropy pairs bounds} only, in order to simplify computations.
 \end{remark}
In particular, when we choose the special test function $\hat{\psi}(s) = \frac{1}{2}s|s|$, for $\rho \geq \rho_*$,
\begin{equation}\label{eq:def eta hat ch2}
    \begin{aligned}
        \hat{\eta}(\rho,u) := \eta^{\hat{\psi}}(\rho,u) &= \eta^{*,\hat{\psi}}(\rho,u) + \tilde{\eta}^{\hat{\psi}}(\rho,u) =: \hat{\eta}^{*}(\rho,u) + \tilde{\eta}(\rho,u).
    \end{aligned}
\end{equation}

\begin{lemma}[Second representation formula for the perturbation]\label{lemma:second rep formula eta error}
    Let $\psi \in C^2(\mathbb{R})$ be such that $\tilde{\eta}^{\psi}_{uu}$, $\eta^{*,\psi}_{u}$, $\eta^{*,\psi}_{uu}$ are well-defined continuous functions. Then, for any $(\rho_0,u_0) \in \tilde{\mathcal{K}}$,
    \begin{equation}\label{eq:rep form eta error from wave}
\begin{aligned}
    2  (\rho_0  & -\rho_*)k'(\rho_0)\tilde{\eta}^{\psi}(\rho_0,u_0) \\
    =&\int_{\rho_*}^{\rho_0} d(\rho) k'(\rho)  \big( \tilde{\eta}^{\psi}(\rho,u_0+k(\rho_0)-k(\rho)) + \tilde{\eta}^{\psi}(\rho,u_0-k(\rho_0)+k(\rho)) \big) \, d\rho \\
+&\int_{\rho_*}^{\rho_0}(\rho-\rho_*)\big(k'(\rho)^2-\frac{\kappa_2}{\rho^2}\big) \eta^{*,\psi}_u(\rho,u_0+k(\rho_0)-k(\rho))  \,d\rho \\
-&\int_{\rho_*}^{\rho_0}(\rho-\rho_*)\big(k'(\rho)^2-\frac{\kappa_2}{\rho^2}\big)\eta^{*,\psi}_u(\rho,u_0-k(\rho_0)+k(\rho)) \,d\rho.
\end{aligned}
\end{equation}
\end{lemma}
The proof of Lemma \ref{lemma:second rep formula eta error} is similar to that of Lemma \ref{lemma:chi rep form}, and makes use of the fact that $\tilde{\eta}^{\psi}$ satisfies an analogous wave equation to \eqref{eq:chi error wave eqn}.

\subsection{Generating entropy fluxes}
Recall from Theorem \ref{thm:expansion at vacuum} that we decompose the entropy flux kernel $\sigma(\rho,u,s)$ for our problem as
\begin{equation}
    \sigma(\rho,u,s) = u \chi (\rho,u,s) + h (\rho,u,s).
\end{equation}
Then the entropy flux corresponding to an entropy $\eta^\psi$ is given by the representation
\begin{equation*}
\begin{aligned}
    q^\psi(\rho,u) = \int_\mathbb{R} \sigma(\rho,u,s) \psi(s) \, ds &= u\int_\mathbb{R} \chi(\rho,u-s) \psi(s) \, ds + \int_\mathbb{R} h(\rho,u-s) \psi(s) \, ds, \\
    &= u\eta^\psi(\rho,u) + \int_\mathbb{R} h(\rho,u-s) \psi(s) \, ds.
    \end{aligned}
\end{equation*}
In turn, defining
\begin{equation}
    H^{\psi}(\rho,u) := \int_\mathbb{R} h(\rho,u-s) \psi(s) \, ds,
\end{equation}
we see that $q^\psi(\rho,u) = u\eta^\psi(\rho,u) + H^\psi(\rho,u)$. As such, any entropy flux can be generated from the kernels $\chi$ and $h$. Analogously to the decomposition of the entropy kernel, we decompose the kernel $h$ as an isothermal kernel and a perturbation. To this end, we make the following definition.
\begin{defn}\label{def:h iso definition and H}
Define $h^{*}$ by
\begin{equation}\label{eq:h iso def}
    h^{*}(\rho,u) := \int_\mathbb{R} h_\rho(\rho_*,s) g^\sharp(\rho,u-s) \, ds + \int_\mathbb{R} h(\rho_*,s) g^\flat(\rho,u-s) \, ds,
\end{equation}
where $g^\sharp$ and $g^\flat$ were introduced in Definition \ref{def:g sharp and g flat} and define also
\beq
\tilde{h}(\rho,u):= h(\rho,u)-h^*(\rho,u) \text{ for }\rho\geq \rho_*.
\eeq
\end{defn}
Note that \eqref{eq:h iso def} and Theorem \ref{thm:expansion at vacuum} imply that $h^{*}$ is the unique solution of
\begin{equation}\label{eq:wave eqn h iso}
    \left\lbrace\begin{aligned}
        &h^{*}_{\rho\rho} - \frac{\kappa_2}{\rho^2} h^{*}_{uu} = 0 \quad \text{for } (\rho,u) \in (\rho_*,\infty)\times \mathbb{R}, \\
        &h^{*}(\rho_*,u)=h(\rho_*,u), \\
        &h^{*}_\rho(\rho_*,u)=h_\rho(\rho_*,u).
    \end{aligned}\right.
\end{equation}
Thus $h^{*}(\rho,\cdot)$ is H\"{o}lder continuous, and
    \begin{equation}
        \supp h^{*}(\rho,\cdot) \subset [-k_*(\rho),k_*(\rho)] \qquad \text{for }(\rho,u)\in[\rho_*,\infty)\times\mathbb{R}.
    \end{equation}
Moreover,
\begin{equation*}
    \sigma^{*}(\rho,u,s) = u\chi^{*}(\rho,u-s) + h^{*}(\rho,u-s),
\end{equation*}
where $\sigma^{*}$ is in fact the entropy flux kernel considered in \cite{SS1} (up to a re-scaling in terms of $\rho_*$).  This observation makes use of the fact, see \cite{SS1}, that the isothermal flux term $h^*$ is itself an entropy.
\begin{lemma}\label{lemma:h iso bounded by rho}
  There exists a positive constant $M$ depending also on $\rho_*$ such that
    \begin{equation}\label{eq:h iso bounded by rho}
        \Vert h^{*}(\rho,\cdot)\Vert_{L^\infty(\mathbb{R})} \leq M\frac{\rho}{\sqrt{k(\rho)}} \qquad \text{for }\rho\geq \rho_*.
    \end{equation}  
\end{lemma}
\begin{proof}
Using the kernels $g^\sharp$ and $g^\flat$ (see Definition \ref{def:g sharp and g flat}), we see that
\begin{equation}
    h^{*}(\rho,u) = \int_\mathbb{R} h_\rho(\rho_*,s) g^\sharp(\rho,u-s) \, ds + \int_\mathbb{R} h(\rho_*,s) g^\flat(\rho,u-s) \, ds.
\end{equation}
The result is now easily deduced using the explicit forms of $g^\sharp$ and $g^\flat$, as was done for $\chi^{*}$ in Proposition \ref{lemma:chi iso rescaled properties}.
\end{proof}
A simple calculation verifies the following lemma.
\begin{lemma}
    The perturbation $\tilde{h}(\rho,u)$ solves, for $(\rho,u) \in (\rho_*,\infty)\times \mathbb{R}$,
\begin{equation}\label{eq:h error wave equation}
\left\lbrace\begin{aligned}
    & \tilde{h}_{\rho \rho} - k'(\rho)^2 \tilde{h}_{uu} = \big( k'(\rho)^2 - \frac{\kappa_2}{\rho^2} \big)h^{*}_{uu}(\rho,u) + \frac{p''(\rho)}{\rho}\chi_u(\rho,u), \\
    & \tilde{h}(\rho_*,u) = 0, \\
    & \tilde{h}_\rho(\rho_*,u) = 0.
\end{aligned}\right.
\end{equation}
\end{lemma}

For $\psi \in C^2_c(\mathbb{R})$ and $\hat{\psi}(s)=\frac{1}{2}s|s|$, we define, for $(\rho,u) \in [\rho_*,\infty)\times\mathbb{R}$,
\begin{equation}
    H^{*,\psi}(\rho,u) := \int_\mathbb{R} h^{*}(\rho,u-s) \psi(s) \, ds, \quad \hat{H}^{*}(\rho,u) := \int_\mathbb{R} h^{*}(\rho,u-s) \hat{\psi}(s) \, ds.
    \end{equation}
Now the entropy flux corresponding to $\eta^\psi$ is given by
$${q}^\psi(\rho,u)={q}^{\psi,*}(\rho,u)+\tilde{q}^\psi(\rho,u)=u{\eta}^{\psi,*}(\rho,u)+u\tilde\eta^\psi(\rho,u)+{H}^{\psi,*}(\rho,u)+\tilde{H}^\psi(\rho,u)$$
and the flux corresponding to $\hat\eta$ is given by
\beq\label{eq:qhat}
\hat{q}(\rho,u)=\hat{q}^*(\rho,u)+\tilde{q}(\rho,u)=u\hat{\eta}^*(\rho,u)+u\tilde\eta(\rho,u)+\hat{H}^*(\rho,u)+\tilde{H}(\rho,u)
\eeq
with obvious notation for $\tilde{H}^\psi$ and $\tilde{H}$.

\vspace{-2mm}

\section{Estimating entropy pairs}\label{section:entropy pairs bounds}
In this section, we use the representation formulas obtained in Section \ref{section:rep formula chi error} and the estimate of Lemma \ref{lemma:error bound abs} to estimate the entropy pairs generated by the special function $\hat{\psi}(s)=\frac{1}{2}s|s|$ and by compactly supported test functions. 
\subsection{Special entropy pair}
We first consider the special entropy, $\hat{\eta}(\rho,u)$, generated by $\hat{\psi}(s) = \frac{1}{2}s|s|$. This was defined by the formula \eqref{eq:def eta hat ch2} for $\rho \geq \rho_*$, 
\begin{equation}\label{eq:special entropy generated}
\begin{aligned}
      \hat{\eta}(\rho,u) &= \frac{1}{2} \int_\mathbb{R} \chi^{*}(\rho,u-s) s|s| \, ds + \frac{1}{2}\int_\mathbb{R} \tilde{\chi}(\rho,u-s) s|s| \, ds,  \\
      & = \hat{\eta}^{*}(\rho,u) + \tilde{\eta}(\rho,u).
\end{aligned}
\end{equation}
The goal of this subsection is to prove the following lemma, which is the analogue of \cite[Lemma 3.5]{SS1}.
\begin{lemma}\label{lemma:pointwise hat chapter 2}
    Let $\hat{\psi}(s)=\frac{1}{2}s|s|$, and $(\hat{\eta},\hat{q})$ its associated entropy pair via \eqref{eq:def eta hat ch2} and \eqref{eq:qhat}. There exists a positive constant $M$ such that
    \begin{equation}\label{ineq:etahatests}
        \begin{aligned}
            &|\hat{\eta}(\rho,m)| \leq M \eta^*(\rho,m), \qquad |\hat{\eta}_m(\rho,m)| \leq M\big( |u| + \sqrt{\log\rho} \big), \qquad |\rho \hat{\eta}_{mm}(\rho,m)| \leq M,
        \end{aligned}
    \end{equation}
    where $\eta^*(\rho,m)$ is the mechanical energy defined in \eqref{eq:physicalentropy} and 
    \beq\label{ineq:qhatlower}
    \hat{q}(\rho,m) \geq M^{-1}\rho|u|^3 - M \big( \rho|u|^2 + \rho + \rho(\log(\rho/\rho_*))^4 \big)
    \eeq
   for all $\rho \geq \rho_*$. For the mixed derivatives $\hat{\eta}_{mu}(\rho,\rho u) = \partial_u \hat{\eta}_m(\rho,\rho u)$ and $\hat{\eta}_{m \rho}(\rho,\rho u) = \partial_\rho \hat{\eta}_m(\rho,\rho u)$, we have 
    \begin{equation}\label{ineq:etahatmixed}
        |\hat{\eta}_{mu}(\rho,\rho u)| \leq M\frac{1}{1+\sqrt{\log(\rho/\rho_*)}}, \qquad |\hat{\eta}_{m\rho}(\rho,\rho u)| \leq M\rho^{-1}.
    \end{equation}
    Moreover, on the complement region $\rho \leq \rho_*$, we have
    \begin{equation}\label{ineq:one}
        \begin{aligned}
            &|\hat{\eta}(\rho,m)| \leq M \eta^*(\rho,m), \qquad && \hat{q}(\rho,m) \geq M^{-1}\big( \rho|u|^3 + \rho^{\gamma+\theta} \big) - M\big( \rho|u|^2 + \rho^\gamma \big), \\
            &|\hat{\eta}_m(\rho,m)| \leq M\big( |u|+\rho^\theta \big), \qquad && |\rho \hat{\eta}_{mm}(\rho,m)| \leq M, \\
            &|\hat{\eta}_{mu}(\rho,\rho u)| \leq M, \qquad && |\hat{\eta}_{m\rho}(\rho,\rho u)| \leq M\rho^{\theta-1},
        \end{aligned}
    \end{equation}
    where in the final line we consider $\hat{\eta}_m(\rho,\rho u)$ as a function of $\rho$ and $u$, as in \eqref{ineq:etahatmixed}. Finally,
    \begin{equation}\label{eq:eta m hat difference squared bounded by e star}
        \rho\big| \hat{\eta}_m(\rho,0) - \hat{\eta}_m(\bar{\rho},0) \big|^2 \leq Me^*(\rho,\bar{\rho}) \qquad \text{for } \rho,\bar{\rho} \geq 0.
    \end{equation}
\end{lemma}

By \cite[Lemma 3.5]{SS1}, the estimates of Lemma \ref{lemma:pointwise hat chapter 2} hold for the principal parts $\hat{\eta}^*$ and $\hat{q}^*$. In addition, the estimates \eqref{ineq:one} in the lower region are by now standard and may be seen, for example, in \cite[Lemma 3.5]{SS1}. Thus we will focus in the sequel on the error terms, $\tilde\eta$ and $\tilde q$. As these estimates are lengthy, the following two subsections will be devoted to their proofs. As a guide to the reader, we point out here that estimates \eqref{ineq:etahatests} may be found in Lemmas \ref{lemma:tildeeta} and \ref{lemma:eta hat error u derivs} while \eqref{ineq:etahatmixed} may be found in Lemma \ref{lemma:eta hat error m rho ch2}. Finally, \eqref{ineq:qhatlower} is proved in Corollary \ref{cor:qhatlower}, while \eqref{eq:eta m hat difference squared bounded by e star} proved similarly to the results in Appendix A (see \cite[Remark 3.51]{thesis} for a complete proof).

\subsubsection{Higher differentiability for the special entropy}
Surprisingly, it is the case that the special entropy $\hat{\eta}$ of \eqref{eq:special entropy generated} is three times continuously differentiable in its second variable, as demonstrated in the next lemma. This property will be used repeatedly in the proofs of \eqref{ineq:etahatests}--\eqref{ineq:etahatmixed}.

\begin{lemma}\label{lemma:three derivs eta hat and eta iso hat}
The entropies $\hat{\eta}^{*}$ and $\tilde{\eta}$ are three times continuously differentiable in their second variable, i.e., $\hat{\eta}^{*}(\rho,\cdot), \tilde{\eta}(\rho,\cdot) \in C^3(\mathbb{R})$ for each $\rho \in [\rho_*,\infty)$. In fact,
\begin{equation}
    \hat{\eta}^{*}_{uuu}(\rho,u) = 2 \chi^{*}(\rho,u), \quad  \tilde{\eta}_{uuu}(\rho,u) = 2\tilde{\chi}(\rho,u) \quad \text{for } (\rho,u) \in [\rho_*,\infty)\times \mathbb{R}.
\end{equation}
Correspondingly, there exists a positive constant $M$ such that
\begin{equation}
    \Vert \hat{\eta}^{*}_{uuu}(\rho,\cdot)\Vert_{L^\infty(\mathbb{R})}\leq \frac{M\rho}{\sqrt{k(\rho)}},  \quad \Vert \tilde{\eta}_{uuu}(\rho,\cdot)\Vert_{L^\infty(\mathbb{R})} \leq M \rho^{1-\alpha \tilde{\lambda}} \quad \text{for } \rho \geq \rho_*.
\end{equation}
\end{lemma}
\begin{proof}
Recall that, by definition,
\begin{equation}\label{eq:expand eta error}
    \tilde{\eta}(\rho,u) = \frac{1}{2}\int_{-\infty}^u \tilde{\chi}(\rho,s) (u-s)^2 \, ds - \frac{1}{2}\int_u^\infty \tilde{\chi}(\rho,s) (u-s)^2 \, ds.
\end{equation}
The rest of the proof now follows from direct calculation and Lemma \ref{lemma:error bound abs}.
\end{proof}

\subsubsection{Bounds on the special entropy and its $u$ derivatives}\label{subsubsection:eta hat u derivs}
Below is the first estimate on the special entropy, which shows that it is controlled by the mechanical energy, $\eta^*$.

\begin{lemma}\label{lemma:tildeeta}
    There exists a positive constant $M$ such that
    \begin{equation}
        |\tilde{\eta}(\rho,u)| \leq M \eta^*(\rho,u).
    \end{equation}
\end{lemma}
\begin{proof}
From Remark \ref{remark:supports}, $\supp \tilde{\chi}(\rho,\cdot) \subset [-\max\{k(\rho),k_*(\rho)\},\max\{k(\rho),k_*(\rho) \}]$. Using the bound  $\max\{k(\rho),k_*(\rho)\}\leq M k(\rho)$ due to Lemma \ref{lemma:bound on max k and k star}, we see from \eqref{eq:expand eta error} and the Cauchy--Schwarz inequality that
\begin{equation}\label{eq:eta error less mech energy}
\begin{aligned}
     |\tilde{\eta}(\rho,u)| &\leq  M k(\rho)u^2 \Vert \tilde{\chi}(\rho,\cdot) \Vert_{L^\infty(\mathbb{R})} + M k(\rho)^3 \Vert \tilde{\chi}(\rho,\cdot) \Vert_{L^\infty(\mathbb{R})}\\
     &\leq \rho^{1-\alpha \tilde{\lambda}}  \big( k(\rho) u^2 + k(\rho)^3 \big)
\end{aligned}
\end{equation}
using Lemma \ref{lemma:error bound abs}. The result now follows easily as $\eta^*(\rho,u)=\half \rho u^2+\rho e(\rho)$ and as $\rho\to\infty$, $\frac{e(\rho)}{\log(\rho)}\to\text{const.}>0$.
\end{proof}
Arguing similarly for the derivatives of $\tilde\eta$, we also obtain the following lemma.
\begin{lemma}\label{lemma:eta hat error u derivs}
    There exists a positive constant $M$ such that
    \begin{equation}
       \begin{aligned}
       & |\tilde{\eta}_m(\rho,u)| \leq \frac{M}{1+\log(\rho/\rho_*)}\big( |u| + 1 \big) && \qquad \text{for } (\rho,u) \in [\rho_*,\infty)\times\mathbb{R},  \\
       & |\tilde{\eta}_{m u}(\rho,u)| + |\rho \tilde{\eta}_{m m}(\rho,u)| \leq \frac{M}{1+\log(\rho/\rho_*)} && \qquad \text{for }(\rho,u)\in[\rho_*,\infty)\times\mathbb{R}.
        \end{aligned}
    \end{equation}
\end{lemma}
The last estimate that we need for $\tilde{\eta}$ is that for $\tilde{\eta}_{m \rho}$. To prove this, we first show an improved estimate on $\tilde{\eta}_m$. This is the subject of the next result, Lemma \ref{lemma:eta error m bound}. Convolving \eqref{eq:chi error wave eqn} with the special test function $\hat{\psi}=\frac{1}{2}s|s|$, we find that $\tilde{\eta}$ satisfies 
\begin{equation}\label{eq:eta special error wave eqn}
    \left\lbrace\begin{aligned}
        &\tilde{\eta}_{\rho\rho} - k'(\rho)^2\tilde{\eta}_{uu} = \big( k'(\rho)^2 - \frac{\kappa_2}{\rho^2} \big) \hat{\eta}^{*}_{uu} \qquad \text{for }(\rho,u)\in(\rho_*,\infty)\times\mathbb{R}, \\
        & \tilde{\eta}(\rho_*,u) = 0 , \\
        & \tilde{\eta}_{\rho}(\rho_*,u) = 0.
    \end{aligned}\right.
\end{equation}
Proceeding exactly as was done in the proof of Lemma \ref{lemma:chi rep form}, we obtain
\begin{equation}\label{eq:rep form h error from wave}
\begin{aligned}
    2(\rho_0-\rho_*) k'(\rho_0) & \tilde{\eta}(\rho_0,u_0) \\=&\int_{\rho_*}^{\rho_0} d(\rho) k'(\rho)  \big( \tilde{\eta}(\rho,u_0+k(\rho_0)-k(\rho)) + \tilde{\eta}(\rho,u_0-k(\rho_0)+k(\rho)) \big) \, d\rho \\
+&\int_{\rho_*}^{\rho_0}(\rho-\rho_*)\big(k'(\rho)^2-\frac{\kappa_2}{\rho^2}\big)   \hat{\eta}^{*}_u(\rho,u_0+k(\rho_0)-k(\rho)) \, d\rho  \\
-&\int_{\rho_*}^{\rho_0}(\rho-\rho_*)\big(k'(\rho)^2-\frac{\kappa_2}{\rho^2}\big)  \hat{\eta}^{*}_u(\rho,u_0-k(\rho_0)+k(\rho))   \, d\rho.
\end{aligned}
\end{equation}
Thus, noting that $\hat{\eta}^*_u=\rho\hat{\eta}^*_m$ and taking a derivative with respect to $u_0$,
\begin{equation}\label{eq:eta hat error m deriv rep formula}
\begin{aligned}
    2(\rho_0-\rho_*) & \rho_0 k'(\rho_0)\tilde{\eta}_m(\rho_0,u_0) \\=&\int_{\rho_*}^{\rho_0} d(\rho) \rho k'(\rho)  \big( \tilde{\eta}_m(\rho,u_0+k(\rho_0)-k(\rho)) + \tilde{\eta}_m(\rho,u_0-k(\rho_0)+k(\rho)) \big) \, d\rho \\
+&\int_{\rho_*}^{\rho_0}(\rho-\rho_*) \rho\big(k'(\rho)^2-\frac{\kappa_2}{\rho^2}\big)  \hat{\eta}^{*}_{mu}(\rho,u_0+k(\rho_0)-k(\rho))   \, d\rho   \\
-&\int_{\rho_*}^{\rho_0}(\rho-\rho_*) \rho\big(k'(\rho)^2-\frac{\kappa_2}{\rho^2}\big)  \hat{\eta}^{*}_{mu}(\rho,u_0-k(\rho_0)+k(\rho))   \, d\rho.
\end{aligned}
\end{equation}

\begin{lemma}\label{lemma:eta error m bound}
    There exists a positive constant $M$ such that
    \begin{equation}
        \Vert \tilde{\eta}_m(\rho,\cdot) \Vert_{L^\infty(\mathbb{R})} \leq M \rho^{-\alpha} \qquad \text{for } \rho \geq \rho_*.
    \end{equation}
\end{lemma}

\begin{proof}
Observe that, after dividing by $2(\rho_0-\rho_*)$, the first term on the right-hand side of \eqref{eq:eta hat error m deriv rep formula} is bounded by
\begin{equation*}
    \frac{1}{(\rho_0-\rho_*)}\int_{\rho_*}^{\rho_0} d(\rho) \Vert \rho k'(\rho) \tilde{\eta}_m(\rho,\cdot) \Vert_{L^\infty(\mathbb{R})} \, d\rho.
\end{equation*}
On the other hand, the sum of the second and third terms on the right-hand side of \eqref{eq:eta hat error m deriv rep formula} is bounded by
\begin{equation*}
\begin{aligned}
     \frac{M}{(\rho_0-\rho_*)}\int_{\rho_*}^{\rho_0}  \big| k'(\rho)^2 - \frac{\kappa_2}{\rho^2} \big| (\rho-\rho_*)\rho \, d\rho,
\end{aligned}
\end{equation*}
where we made use of the bound on $\hat{\eta}^{*}_{mu}$ provided by Proposition \ref{lemma:chi iso rescaled properties} (\textit{cf.}~\cite[Lemma 3.5]{SS1}). Thus, appealing to Corollary \ref{corollary:k' and k'' compare with rho powers} to control the $k'(\rho)^2$ term, the sum of the second and third terms on the right-hand side of \eqref{eq:eta hat error m deriv rep formula} is bounded by
\begin{equation*}
    \frac{M}{\rho_0-\rho_*}\int_{\rho_*}^{\rho_0}\rho^{-\alpha} \, d\rho \leq M\rho_0^{-\alpha},
\end{equation*}
provided $\alpha \in (0,1)$. In summary, we get 
\begin{equation*}
\begin{aligned}
    \Vert \rho_0 k'(\rho_0) \tilde{\eta}_m(\rho_0,\cdot) \Vert_{L^\infty(\mathbb{R})} \leq & M \rho_0^{-\alpha} +  \frac{1}{(\rho_0-\rho_*)}\int_{\rho_*}^{\rho_0} d(\rho) \Vert \rho k'(\rho) \tilde{\eta}_m(\rho,\cdot) \Vert_{L^\infty(\mathbb{R})} \, d\rho,
    \end{aligned}
\end{equation*}
from which an application of Gr\"{o}nwall's lemma yields
\begin{equation*}
    \Vert \rho_0 k'(\rho_0) \tilde{\eta}_m(\rho_0,\cdot) \Vert_{L^\infty(\mathbb{R})} \leq M \rho_0^{-\alpha} \exp\left( \frac{1}{\rho_0-\rho_*}\int_{\rho_*}^{\rho_0} d(\rho) \, d\rho \right).
\end{equation*}
Using the bound on $d(\rho)$ provided by Lemma \ref{lemma:d and d star} concludes the proof.
\end{proof}

Armed with the previous result, we are in a position to prove the required estimate on the mixed derivative $\tilde{\eta}_{m \rho}$, contained in the next lemma.
\begin{lemma}\label{lemma:eta hat error m rho ch2}
    There exists a positive constant $M$ such that
    \begin{equation}
        \Vert \tilde{\eta}_{m \rho}(\rho,\cdot) \Vert_{L^\infty(\mathbb{R})} \leq M k(\rho) \rho^{-1-\alpha \tilde{\lambda}} \qquad \text{for } \rho \geq \rho_*.
    \end{equation}
\end{lemma}

\begin{proof}
We begin by differentiating \eqref{eq:eta hat error m deriv rep formula} with respect to $\rho_0$. We get
\begin{equation}\label{eq:hat eta m rho rep}
\begin{aligned}
    2 (\rho_0&-\rho_* ) \rho_0 k'(\rho_0)\tilde{\eta}_{m \rho}(\rho_0,u_0) + \partial_{\rho_0}\big(2(\rho_0-\rho_*) \rho_0 k'(\rho_0)\big)\tilde{\eta}_m(\rho_0,u_0) \\
   = \,& k'(\rho_0)\int_{\rho_*}^{\rho_0} d(\rho) \rho k'(\rho)  \big( \tilde{\eta}_{mu}(\rho,u_0+k(\rho_0)-k(\rho)) - \tilde{\eta}_{mu}(\rho,u_0-k(\rho_0)+k(\rho)) \big) \, d\rho \\
    +&\, 2 d(\rho_0) \rho_0 k'(\rho_0)  \tilde{\eta}_m(\rho_0,u_0) \\
+& \,k'(\rho_0)\int_{\rho_*}^{\rho_0}(\rho-\rho_*) \rho\big(k'(\rho)^2-\frac{\kappa_2}{\rho^2}\big)  \hat{\eta}^{*}_{muu}(\rho,u_0+k(\rho_0)-k(\rho))  \,d\rho   \\
+&\, k'(\rho_0)\int_{\rho_*}^{\rho_0}(\rho-\rho_*) \rho\big(k'(\rho)^2-\frac{\kappa_2}{\rho^2}\big)   \hat{\eta}^{*}_{muu}(\rho,u_0-k(\rho_0)+k(\rho)) \,d\rho.
\end{aligned}
\end{equation}
Using the bounds of Lemmas \ref{lemma:eta hat error u derivs} and \ref{lemma:eta error m bound}, we may apply a Gr\"{o}nwall argument, similarly to the proof of Lemma \ref{lemma:eta error m bound} in order to conclude the desired result. Note that the second term on the left-hand side of \eqref{eq:hat eta m rho rep} is controlled by expanding the derivative term and using the bounds provided by Lemmas \ref{lemma:p' below} and \ref{lemma:eta error m bound}, and Corollary \ref{corollary:k' and k'' compare with rho powers}. Further details may be found in \cite[Chapter 3]{thesis}.
\end{proof}

We conclude this subsection with a stronger estimate on the entropy error, $\tilde{\eta}$.
\begin{cor}\label{corollary:better hat eta error}
There exists a positive constant $M$ such that
\begin{equation}
    |\tilde{\eta}(\rho,u)| \leq M\rho |u| \qquad \text{for } (\rho,u)\in[\rho_*,\infty)\times\mathbb{R}.
\end{equation}
\end{cor}
\begin{proof}
Firstly, recall from the definition of $\tilde{\eta}$ that
\begin{equation}\label{eq:hat eta error is 0 at u is 0}
    \tilde{\eta}(\rho,0) = \frac{1}{2}\int_\mathbb{R} \tilde{\chi}(\rho,0-s)s|s| \, ds = 0 \qquad \text{for }\rho\geq \rho_*,
\end{equation}
since the integrand is odd, as $\tilde{\chi}(\rho,\cdot)$ is even (\textit{cf.}~\cite[Lemma 2.7]{thesis}). We also have 
\begin{equation}\label{eq:hat eta error m rho star is 0}
    \tilde{\eta}_m(\rho_*,u) = \rho_*^{-1}\frac{\partial}{\partial u} \tilde{\eta}(\rho_*,u) = 0 \qquad \text{for all } u \in \mathbb{R}.
\end{equation}
Additionally, since $k(\rho) \leq M( 1 + \log(\rho/\rho_*))$ by Corollary \ref{corollary:k sim log}, there exists $M>0$ such that
\begin{equation*}
    \rho^{-\alpha \tilde{\lambda}/2}k(\rho) \leq M \qquad \text{for } \rho \geq \rho_*.
\end{equation*}
The fundamental theorem of calculus and \eqref{eq:hat eta error m rho star is 0} show that
\begin{equation*}
    \tilde{\eta}_m(\rho,u) - \tilde{\eta}_m(\rho_*,u) = \int_{\rho_*}^\rho \frac{\partial}{\partial \rho}\bigg|_u \tilde{\eta}_m(y,u) \, dy = \int_{\rho_*}^\rho \tilde{\eta}_{m \rho}(y,u) \, dy,
\end{equation*}
where the final integrand is the mixed derivative $\partial_\rho \hat{\eta}_m(\rho, \rho u)$. Therefore, integrating in $\rho$ yields
\begin{equation*}
\begin{aligned}
       |\rho^{-1}\tilde{\eta}_u(\rho,u)|= |\tilde{\eta}_{m}(\rho,u)| \leq \int^\rho_{\rho_*} |\tilde{\eta}_{m\rho}(y,u)| \, dy  &\leq M\int_{\rho_*}^\rho y^{-1-\alpha \tilde{\lambda}} k(y) \, dy, \\
       &\leq M\int_{\rho_*}^\rho y^{-1-\alpha \tilde{\lambda}/2} \, dy,
\end{aligned}
\end{equation*}
 using Lemma \ref{lemma:eta hat error m rho ch2},  so that $\tilde{\eta}_u(\rho,u) \leq M\rho$ for all $\rho\geq\rho_*$. Integrating the above in $u$, using the fact that  $\tilde{\eta}(\rho,0)=0$ from \eqref{eq:hat eta error is 0 at u is 0}, yields the result.
\end{proof}

\subsubsection{Lower bound for the special entropy flux}\label{subsubsection:lower bound hat q}
The final estimate of Lemma \ref{lemma:pointwise hat chapter 2} that we need to prove is \eqref{ineq:qhatlower}. This is the purpose of this subsection. We recall from \eqref{eq:qhat} that
$$\hat{q}^*(\rho,u)=u\hat{\eta}^*(\rho,u)+\hat{H}^*(\rho,u)$$
and begin this subsection by collecting some estimates on 
 $\hat{H}^{*}$.
\begin{lemma}\label{lemma:H iso C3}
The function $\hat{H}^{*}$ is three times continuously differentiable in its second variable, i.e., $\hat{H}^{*}(\rho,\cdot) \in C^3(\mathbb{R})$. In fact,
\begin{equation}
    \hat{H}^{*}_{uuu}(\rho,u) = 2h^{*}(\rho,u) \qquad \text{for } (\rho,u) \in [\rho_*,\infty)\times\mathbb{R}.
\end{equation}
In addition, there exists $M>0$ depending on $\rho_*$ such that
\begin{equation}\label{eq:blunt H hat iso uu}
\begin{aligned}
     \Vert \hat{H}^{*}_{uu}(\rho,\cdot)\Vert_{L^\infty(\mathbb{R})}+ k(\rho)\Vert \hat{H}^{*}_{uuu}(\rho,\cdot)\Vert_{L^\infty(\mathbb{R})} \leq M{\rho}\sqrt{k(\rho)}\qquad \text{for } \rho \geq \rho_*.
\end{aligned}
\end{equation}
\end{lemma}
The proof of Lemma \ref{lemma:H iso C3} is similar to that of Lemma \ref{lemma:three derivs eta hat and eta iso hat}.

In order to bound $\tilde{q}(\rho,u)$, we require an estimate on $\tilde{H}(\rho,u)$. This is the main content of the following lemma.
\begin{lemma}\label{lemma:H error hat bound}
There exists $M>0$ such that, for all  $(\rho,u) \in [\rho_*,\infty) \times \mathbb{R}$,
\begin{equation}\label{ineq:tildeHmainest}
    |\tilde{H}(\rho,u)| \leq M\big( \rho|u| + \rho + \rho\sqrt{\log(\rho/\rho_*)} \big).
\end{equation}
In addition, for all $\rho\geq \rho_*$, we have the bounds
\beqa\label{ineq:tildeHderivs}
&\Vert \tilde{H}_u(\rho,\cdot)\Vert_{L^\infty(\mathbb{R})} \leq M k(\rho) \rho^{1-\alpha} , \\
  &  \Vert k'(\rho) \tilde{H}_{uu}(\rho,\cdot) \Vert_{L^\infty(\mathbb{R})} \leq M \rho^{-\alpha} .
\eeqa
\end{lemma}

\begin{proof}
We begin with the estimates on the derivatives of $\tilde{H}$ and subsequently apply these to deduce the main bound for $\tilde{H}$.

 By convolving \eqref{eq:h error wave equation} with the special function $\hat{\psi}(s) =\frac{1}{2}s|s|$, we obtain
\begin{equation}\label{eq:H error wave equation}
\left\lbrace\begin{aligned}
    & \tilde{H}_{\rho \rho} - k'(\rho)^2 \tilde{H}_{uu} = \big( k'(\rho)^2 - \frac{\kappa_2}{\rho^2} \big)\hat{H}^{*}_{uu}(\rho,u) + \frac{p''(\rho)}{\rho}\hat{\eta}_u(\rho,u) , \\
    & \tilde{H}(\rho_*,u) = 0, \\
    & \tilde{H}_\rho(\rho_*,u) = 0.
\end{aligned}\right.
\end{equation}
Therefore by arguments similar to those in Lemma \ref{lemma:chi rep form} (see also Lemma \ref{lemma:second rep formula eta error}), we find the following representation formula for $\tilde{H}$:
\begin{equation}\label{eq:rep form h error}
\begin{aligned}
   2(\rho_0-\rho_*)  k'(\rho_0) & \tilde{H}(\rho_0,u_0) \\ =&\,\int_{\rho_*}^{\rho_0} d(\rho) k'(\rho)  \big( \tilde{H}(\rho,u_0+k(\rho_0)-k(\rho)) + \tilde{H}(\rho,u_0-k(\rho_0)+k(\rho)) \big) \, d\rho \\
+&\int_{\rho_*}^{\rho_0}(\rho-\rho_*)\big(k'(\rho)^2-\frac{\kappa_2}{\rho^2}\big)  \hat{H}^{*}_u(\rho,u_0+k(\rho_0)-k(\rho))  \,d\rho\\
-&\int_{\rho_*}^{\rho_0}(\rho-\rho_*)\big(k'(\rho)^2-\frac{\kappa_2}{\rho^2}\big)  \hat{H}^{*}_u(\rho,u_0-k(\rho_0)+k(\rho))\,d\rho  \\
         +&\int_{\rho_*}^{\rho_0}\frac{p''(\rho)}{\rho}(\rho-\rho_*)\big(\hat{\eta}(\rho,u_0+k(\rho_0)-k(\rho))-\hat{\eta}(\rho,u_0-k(\rho_0)+k(\rho))\big)\,d\rho,
\end{aligned}
\end{equation}
which has a similar structure to \eqref{eq:rep form h error from wave}. The proof of estimates \eqref{ineq:tildeHderivs} now follows the same lines as the proof of Lemma \ref{lemma:eta error m bound}, and so we omit the details.

To conclude the final estimate, \eqref{ineq:tildeHmainest}, we convolve \eqref{eq:h error wave equation} with the function $\hat{\psi}=\half s|s|$, to obtain
\begin{equation*}
    \tilde{H}_{\rho \rho}(\rho,u) = k'(\rho)^2 \tilde{H}_{uu}(\rho,u) + \big( k'(\rho)^2 - \frac{\kappa_2}{\rho^2}  \big) \hat{H}^{*}_{uu}(\rho,u) + \frac{p''(\rho)}{\rho}\hat{\eta}_u(\rho,u),
\end{equation*}
from which we may directly estimate, using \eqref{eq:k'' close to rho squared} and  \eqref{eq:blunt H hat iso uu}, and the bound on $\hat{\eta}_u$ provided by Lemmas \ref{lemma:chi iso rescaled properties} (\textit{cf.}~\cite[Lemma 3.5]{SS1}) and \ref{lemma:eta hat error u derivs},
\begin{equation*}
    |\tilde{H}_{\rho \rho}(\rho,u)| \leq M \bigg( \rho^{-\alpha-1}  + \rho^{-\alpha-1} \big( |u| + \sqrt{k(\rho)} \big) \bigg),
\end{equation*}
We now integrate in $\rho$, making use of the fact that $\tilde{H}_\rho(\rho_*,\cdot) = 0$ and the monotonicity of $k$, and obtain
\begin{equation*}
\begin{aligned}
     |\tilde{H}_\rho(\rho,u)| \leq M\bigg( \rho_*^{-\alpha} &+ \rho_*^{-\alpha} \big(|u| 
     + \sqrt{\log(\rho/\rho_*)} \big) \bigg).
\end{aligned}
\end{equation*}
Integrating once more in $\rho$, and making use of the fact that $\tilde{H}(\rho_*,\cdot)=0$, gives the result.
\end{proof}

Using this result, we are able to bound $\tilde{q}$ in absolute value.
\begin{cor}\label{corollary:upper bound hat q error}
    There exists a positive constant $M$ such that
    \begin{equation}
        |\tilde{q}(\rho,u)| \leq M\big( \rho|u|^2 + \rho + \rho \sqrt{\log (\rho/\rho_*)} \big) \qquad \text{for } (\rho,u)\in[\rho_*,\infty)\times\mathbb{R}.
    \end{equation}
\end{cor}
\begin{proof}
Given the decomposition for $\tilde{q}$ given in \eqref{eq:qhat},
\begin{equation*}
\begin{aligned}
    |\tilde{q}(\rho,u)| &\leq |u\tilde{\eta}(\rho,u)| + |\tilde{H}(\rho,u)|, \\
    &\leq M\big( \rho|u|^2 + \rho + \rho\sqrt{\log(\rho/\rho_*)} \big),
\end{aligned}
\end{equation*}
where we used the results of Corollary \ref{corollary:better hat eta error}, Lemma \ref{lemma:H error hat bound}, and the Cauchy--Schwarz inequality.
\end{proof}
We now state the lower bound on the entropy flux, \eqref{ineq:qhatlower}, from Lemma \ref{lemma:pointwise hat chapter 2}.
\begin{cor}\label{cor:qhatlower}
There exists a positive constant $M$ such that
\begin{equation}
    \hat{q}(\rho,u) \geq M^{-1} \rho|u|^3 - M\big( \rho |u|^2 + \rho + \rho(\log (\rho/\rho_*))^4 \big) \qquad \text{for }(\rho,u) \in [\rho_*,\infty)\times\mathbb{R}.
\end{equation}
\end{cor}
\begin{proof}
Recall that, as we saw in \eqref{eq:qhat}, we may write $ \hat{q}(\rho,u) = \hat{q}^{*}(\rho,u) + \tilde{q}(\rho,u)$, where $\hat{q}^{*}$ satisfies the lower bound \eqref{ineq:qhatlower} by \cite[Lemma 3.5]{SS1}. The result follows from Corollary \ref{corollary:upper bound hat q error}.
\end{proof}

\begin{remark}
The results of Subsections \ref{subsubsection:eta hat u derivs} and \ref{subsubsection:lower bound hat q} prove Lemma \ref{lemma:pointwise hat chapter 2} in full except for \eqref{eq:eta m hat difference squared bounded by e star}, which is contained in Appendix \ref{section:appendix} and written fully in \cite[Section 3.4]{thesis}.
\end{remark}

\subsection{Entropies generated by compactly supported test functions}
We now consider entropies generated by functions $\psi \in C^2_c(\mathbb{R})$. Recall from \eqref{eq:generate any psi entropy} that, for $\rho \geq \rho_*$,
\begin{equation*}
\begin{aligned}
      \eta^\psi(\rho,u) &= \int_\mathbb{R} \chi^{*}(\rho,u-s) \psi(s) \, ds + \int_\mathbb{R} \tilde{\chi}(\rho,u-s) \psi(s) \, ds , \\
      & = \eta^{*,\psi}(\rho,u) + \tilde{\eta}^{\psi}(\rho,u).
\end{aligned}
\end{equation*}

The goal of this subsection is to prove the following lemma, which is the analogue of \cite[Lemma 3.8]{SS1}.

\begin{lemma}\label{lemma:cpctly supported chapter 2}
    Let $\psi \in C^2_c(\mathbb{R})$ and $(\eta^\psi,q^\psi)$ the associated entropy pair via \eqref{eq:generate entropy} and \eqref{eq:generate entropy flux}. There exists a positive constant $M_\psi$ such that
    \begin{equation}
        \begin{aligned}
            |\eta^\psi(\rho,m)| \leq M_\psi \rho\min\{1, \frac{1}{\sqrt{\log (\rho/\rho_* + 1)}}\}, \quad |q^\psi(\rho,m)| \leq M_\psi \rho  \quad \text{for } (\rho,m) \in \mathbb{R}^2_+.
        \end{aligned}
        \end{equation}
        Also, with $\eta^\psi_{mu}(\rho,\rho u) = \partial_u \eta^\psi_m(\rho,\rho u)$ the usual mixed derivative,
        \begin{equation}
            |\eta_m^\psi(\rho,m)| + |\eta_{mu}^\psi(\rho,m)| + |\rho \eta^\psi_{mm}(\rho,m)| \leq M_\psi \min\{1, \frac{1}{\sqrt{\log (\rho/\rho_* + 1)}}\}  \quad \text{for }(\rho,m) \in \mathbb{R}^2_+.
        \end{equation}
        Finally, with $\eta^\psi_{m\rho}(\rho,\rho u) = \partial_\rho \eta^\psi_m(\rho,\rho u)$ the usual mixed derivative,
        \begin{equation}
            |\eta_{m \rho}^\psi(\rho,m)| \leq M_\psi \frac{\sqrt{p'(\rho)}}{\rho} \qquad \text{for }(\rho,m) \in \mathbb{R}^2_+.
        \end{equation}
\end{lemma}
We recall that \cite[Lemma 3.8]{SS1} proved these inequalities for the principal parts, $\eta^{\psi,*}$ and $q^{\psi,*}$. We therefore focus on the error terms, $\tilde{\eta}^\psi$ and $\tilde{q}^\psi$.
To begin with, in view of Lemma \ref{lemma:error bound abs}, we can bound the error as
\begin{equation}\label{eq:eta error psi compact bounding procedure}
    \begin{aligned}
    |\tilde{\eta}^{\psi}(\rho,u)| &\leq M \rho^{1-\alpha \tilde{\lambda}} \int_\mathbb{R} |\psi(s)|ds, \\
    &\leq M \rho^{1-\alpha \tilde{\lambda}}.
    \end{aligned}
\end{equation}
Notice also that, if $\supp \psi \subset [a,b]$, then, in view of \eqref{eq:supp chi error},
\begin{equation*}
\begin{aligned}
    \tilde\eta^{\psi}(\rho,u) = \int_\mathbb{R} \tilde{\chi}(\rho,s) \psi(u-s) \, ds = \int_{-\max\{ k(\rho),k_*(\rho) \}}^{\max\{ k(\rho),k_*(\rho) \}} \tilde{\chi}(\rho,s) \psi(u-s) \, ds,
    \end{aligned}
\end{equation*}
and the right-hand side vanishes for $u \notin [a-\max\{k(\rho),k_*(\rho)\},b+\max\{k(\rho),k_*(\rho)\}]$. Additionally,  differentiating in $u$, we get
\begin{equation*}
    \tilde{\eta}_u^{\psi}(\rho,u) = \int_\mathbb{R} \tilde{\chi}(\rho,u-s) \psi'(s) \, ds.
\end{equation*}
Taking further derivatives and arguing as in \eqref{eq:eta error psi compact bounding procedure}, it easily follows from Lemma \ref{lemma:error bound abs} that
\begin{equation*}
    |\tilde{\eta}^{\psi}_m(\rho,u)| + |\tilde{\eta}^{\psi}_{mu}(\rho,u)| + |\rho \tilde{\eta}^{\psi}_{mm}(\rho,u)| \leq M \rho^{-\alpha \tilde{\lambda}},
\end{equation*}
where $\tilde{\eta}^{\psi}_{mu}(\rho,\rho u) = \partial_u \tilde{\eta}^\psi_m(\rho,\rho u)$ is the usual mixed derivative. Hence, we have proved the following lemma.
\begin{lemma}\label{lemma:eta error u bounds and support}
    Let $\psi \in C^2_c(\mathbb{R})$ be such that $\supp \psi \subset [a,b]$. Then,
    \begin{equation}
        \supp \tilde{\eta}^{\psi}(\rho,\cdot) \subset [a-\max\{k(\rho),k_*(\rho)\},b+\max\{k(\rho),k_*(\rho)\}].
    \end{equation}
    Also, there exists a positive constant $M_\psi$ such that
    \begin{equation}\label{eq:eta error cpct bound}
        \Vert \tilde{\eta}^{\psi}(\rho,\cdot) \Vert_{L^\infty(\mathbb{R})} +  \Vert \rho \tilde{\eta}^{\psi}_m(\rho,\cdot) \Vert_{L^\infty(\mathbb{R})}\leq M_\psi \rho^{1-\alpha \tilde{\lambda}} \qquad \text{for } \rho \geq \rho_*,
    \end{equation}
    and
    \begin{equation}
    \Vert \tilde{\eta}^{\psi}_{mu}(\rho,\cdot)\Vert_{L^\infty(\mathbb{R})} + \Vert \rho \tilde{\eta}^{\psi}_{mm}(\rho,\cdot)\Vert_{L^\infty(\mathbb{R})} \leq M_\psi \rho^{-\alpha \tilde{\lambda}} \qquad \text{for } \rho \geq \rho_*.
\end{equation}
\end{lemma}

\begin{cor}\label{cor:u eta error psi}
For $\psi \in C^2_c(\mathbb{R})$, there exists a positive constant $M_\psi $ such that
\begin{equation}
    |u\tilde{\eta}^{\psi}(\rho,u)| \leq M_\psi k(\rho) \rho^{1-\alpha \tilde{\lambda}} \qquad \text{for } (\rho,u)\in[\rho_*,\infty)\times\mathbb{R}.
\end{equation}
\end{cor}
\begin{proof}
The result now follows easily from Lemma \ref{lemma:eta error u bounds and support}, using the compact support of $\tilde{\eta}^{\psi}$ and the estimate \eqref{eq:eta error cpct bound}, along with the bound on $\max\{k(\rho),k_*(\rho)\}$ provided by Lemma \ref{lemma:bound on max k and k star}.
\end{proof}
On the other hand, \cite[Lemma 3.8]{SS1} showed that 
\begin{equation*}
    |\eta^{*,\psi}(\rho,u)| \leq \frac{M_\psi \rho}{\sqrt{\log( \rho/\rho_* + 1)}} \qquad \text{for } \rho \geq \rho_*.
\end{equation*}
Thus, adding the contributions from $\eta^{*,\psi}$ and $\tilde{\eta}^{\psi}$,
\begin{equation}\label{eq:prelim error cpct m derivs}
    \Vert \eta^{\psi}(\rho,\cdot) \Vert_{L^\infty(\mathbb{R})} \leq \frac{M_\psi \rho}{\sqrt{k(\rho)}} \qquad \text{for }\rho \geq \rho_*.
\end{equation}
We now consider the mixed derivative of the full entropy, $\eta^{\psi}_{m \rho}$, for which the following result holds.
\begin{lemma}\label{lemma:eta rho m cpct test}
   Let $\psi \in C^2_c(\mathbb{R})$ and $\eta^\psi$ be generated by \eqref{eq:generate entropy}. Then, there exists a positive constant $M_\psi$ such that
\begin{equation}
    \Vert\eta^{\psi}_{m \rho}(\rho,\cdot)\Vert_{L^\infty(\mathbb{R})} \leq M_\psi \frac{\sqrt{p'(\rho)}}{\rho} \qquad \text{for } \rho \geq \rho_*.
\end{equation}
\end{lemma}

\begin{proof}
We first use the representation formula \eqref{eq:rep form chi} to write down
\begin{equation*}
\begin{aligned}
\eta^{\psi}_m(\rho_0,u_0) = &\, \frac{1}{2\rho_0(\rho_0 - \rho_*)k'(\rho_0)}\int_{\rho_*}^{\rho_0} k'(\rho) d(\rho)  \eta^{\psi}_u(\rho,u_0+k(\rho_0)-k(\rho)) \, d\rho \\
&+ \frac{1}{2\rho_0(\rho_0 - \rho_*)k'(\rho_0)}\int_{\rho_*}^{\rho_0} k'(\rho) d(\rho) \eta^{\psi}_u(\rho,u_0-k(\rho_0)+k(\rho)) \, d\rho \\
& - \frac{1}{2\rho_0(\rho_0 - \rho_*)k'(\rho_0)} \int_{-(k(\rho_0)-k(\rho_*))}^{k(\rho_0)-k(\rho_*)} \eta^{\psi}_u(\rho_*,u_0-y) \, dy, \\
=: &\, I + II + III.
\end{aligned}
\end{equation*}
where $\eta^{\psi}_u$ is given by $\eta_u^\psi(\rho,u) = \int_\mathbb{R} \chi(\rho,u-s) \psi'(s) \, ds$. Hence, the mixed derivative $\eta^{\psi}_{m \rho}$ is given by
\begin{equation*}
    \eta^{\psi}_{m \rho}(\rho_0,u_0) = \partial_{\rho_0}I + \partial_{\rho_0}II + \partial_{\rho_0}III.
\end{equation*}
We begin by controlling the term $\partial_{\rho_0}I$.
\begin{equation*}
\begin{aligned}
    \partial_{\rho_0}I = & \,\partial_{\rho_0}\left( \frac{1}{2\rho_0(\rho_0-\rho_*)k'(\rho_0)} \right) \int_{\rho_*}^{\rho_0} k'(\rho) d(\rho) \eta^{\psi}_u(\rho,u_0+k(\rho_0)-k(\rho)) \, d\rho \\
    & + \frac{1}{2\rho_0(\rho_0-\rho_*)}d(\rho_0)\eta^{\psi}_u(\rho_0,u_0)  + \frac{1}{2\rho_0(\rho_0-\rho_*)} \int_{\rho_*}^{\rho_0} k'(\rho) d(\rho)  \eta^{\psi}_{uu}(\rho,u_0+k(\rho_0)-k(\rho)) \, d\rho, \\
    =: &\, I_1 + I_2 + I_3.
\end{aligned}
\end{equation*}
We expand the first term in big brackets,
\begin{equation*}
\begin{aligned}
    \partial_{\rho_0}\left( \frac{1}{2\rho_0(\rho_0-\rho_*)k'(\rho_0)} \right) = -\frac{1}{2\rho_0^2 (\rho_0-\rho_*) k'(\rho_0)} &- \frac{1}{\rho_0(\rho_0-\rho_*)^2 k'(\rho_0)} - \frac{k''(\rho_0)}{2\rho_0(\rho_0-\rho_*)k'(\rho_0)^2}.
\end{aligned}
\end{equation*}
Therefore, using Lemma \ref{lemma:p' below}, $\left| \partial_{\rho_0}\left( \frac{1}{2\rho_0(\rho_0-\rho_*)k'(\rho_0)} \right) \right| \leq \frac{M}{(\rho_0-\rho_*)^2}$. Thus, using \cite[Lemma 3.8]{SS1} and Lemma \ref{lemma:eta error u bounds and support},
\begin{equation*}
    \begin{aligned}
        |I_1|\leq \frac{M}{(\rho_0-\rho_*)^2}\int_{\rho_*}^{\rho_0} \, d\rho = \frac{M}{\rho_0-\rho_*}.
    \end{aligned}
\end{equation*}
For $I_2$, we have $\left| \frac{1}{2\rho_0(\rho_0-\rho_*)}d(\rho_0) \eta^{\psi}_u(\rho_0,u_0) \right| \leq \frac{M}{\rho_0-\rho_*}$. The last term in the expansion of $\partial_{\rho_0}I$ is controlled as
\begin{equation*}
    \begin{aligned}
    |I_3| &\leq \frac{M}{\rho_0(\rho_0-\rho_*)}\int_{\rho_*}^{\rho_0}  \, d\rho \leq \frac{M}{\rho_0-\rho_*}.
    \end{aligned}
    \end{equation*}
    Thus, for $\rho_0\geq \rho_*+1$,
    \begin{equation*}
    |\partial_{\rho_0} I| \leq \frac{M}{\rho_0-\rho_*}\leq \frac{M}{1+(\rho_0-\rho_*)}.
\end{equation*}
On the other hand, it is straightforward to check that $\eta^{\psi}_{m\rho}$ also remains bounded for $\rho_0\in(\rho_*,\rho_*+1)$. Hence, we conclude that
\begin{equation*}
    |\partial_{\rho_0} I| \leq  \frac{M}{1+(\rho_0-\rho_*)},
\end{equation*}
and the same bound holds for $|\partial_{\rho_0}II|$. We now bound the final term,
\begin{equation*}
\begin{aligned}
    \partial_{\rho_0}III = &-\partial_{\rho_0}\left( \frac{1}{2\rho_0(\rho_0-\rho_*)k'(\rho_0)} \right) \int_{-(k(\rho_0)-k(\rho_*))}^{k(\rho_0)-k(\rho_*)} \eta^{\psi}_u(\rho_*,u_0-y) \, dy \\
    &-\frac{1}{2\rho_0(\rho_0-\rho_*)} \big( \eta_u^{\psi}(\rho_*,u_0-k(\rho_0)+k(\rho_*)) + \eta_u^{\psi}(\rho_*,u_0+k(\rho_0)-k(\rho_*)) \big).
\end{aligned}
\end{equation*}
Hence, using \eqref{eq:prelim error cpct m derivs}, we get
\begin{equation*}
    |\partial_{\rho_0}III| \leq \frac{M\rho_*}{(\rho_0-\rho_*)^2}2(k(\rho_0)-k(\rho_*)) + \frac{M\rho_*}{\rho_0 (\rho_0-\rho_*)} \leq \frac{M}{\rho_0-\rho_*}.
\end{equation*}
Thus, checking again that no singularity arises at $\rho_0=\rho_*$,
\begin{equation*}
    |\partial_{\rho_0}III| \leq \frac{M}{1+(\rho_0-\rho_*)}.
\end{equation*}
Meanwhile, using Lemma \ref{lemma:p' below}, we have $  \sqrt{p'(\rho)} \rho^{-1} \geq  \sqrt{\frac{\kappa_2}{2}}\rho^{-1}$. The result follows at once.
\end{proof}

In fact, the same procedure can be followed so as to obtain the same estimates for $\eta^{*,\psi}_{m \rho}$, from which we see that $|\tilde{\eta}^{\psi}_{m\rho}(\rho_0,u_0)| \leq M\rho_0^{-1}$. In view of this, we have the following corollary.
\begin{cor}\label{corollary:eta rho m cpct test cor}
    There exists a positive constant $M_\psi$ such that
\begin{equation}
    \Vert\eta^{*,\psi}_{m \rho}(\rho,\cdot)\Vert_{L^\infty(\mathbb{R})} + \Vert\tilde{\eta}^{\psi}_{m \rho}(\rho,\cdot)\Vert_{L^\infty(\mathbb{R})} \leq M_\psi \frac{\sqrt{p'(\rho)}}{\rho} \qquad \text{for } \rho \geq \rho_*.
\end{equation}
\end{cor}

We have yet to inspect the entropy flux. To this end, we recall the equation for $\tilde{H}^{\psi}$, obtained by convolving \eqref{eq:h error wave equation} with $\psi \in C^2_c(\mathbb{R})$,
\begin{equation}\label{eq:H error psi wave equation}
\left\lbrace\begin{aligned}
    & \tilde{H}_{\rho \rho}^{\psi} - k'(\rho)^2 \tilde{H}_{uu}^{\psi} = f^\psi(\rho,u), \quad \text{for } (\rho,u) \in (\rho_*,\infty)\times \mathbb{R}, \\
    & \tilde{H}^{\psi}(\rho_*,u) = 0, \\
    & \tilde{H}_\rho^{\psi}(\rho_*,u) = 0,
\end{aligned}\right.
\end{equation}
where
\begin{equation}
    f^\psi(\rho,u) := \big( k'(\rho)^2 - \frac{1}{\rho^2} \big)H^{*,\psi}_{uu}(\rho,u) + \frac{p''(\rho)}{\rho}\eta^\psi_u(\rho,u) \qquad \text{for } (\rho,u) \in [\rho_*,\infty)\times\mathbb{R}.
\end{equation}
This equation allows us to make a similar representation formula and Gr\"{o}nwall argument to those above to obtain the following lemma, the proof of which we omit.

\begin{lemma}\label{lemma:H error psi}
    There exists a positive constant $M_\psi$ such that
    \begin{equation}
        \Vert \tilde{H}^{\psi}(\rho,\cdot)\Vert_{L^\infty(\mathbb{R})} \leq M_\psi k(\rho) \rho^{1-\alpha} \qquad \text{for }(\rho,u)\in[\rho_*,\infty)\times\mathbb{R}.
    \end{equation}
\end{lemma}

By recalling $\tilde{q}^\psi = u\tilde{\eta}^\psi + \tilde{H}^\psi$ and combining Corollary \ref{cor:u eta error psi} and Lemma \ref{lemma:H error psi} to bound the entropy flux, the proof of Lemma  \ref{lemma:cpctly supported chapter 2} is now complete.

\section{Uniform estimates and compactness of the entropy dissipation measures}\label{section:uniformestimates}
Finally, we prove the uniform estimates that are required for the compactness of the entropy dissipation measures. To begin with, we collect the assumptions required on the initial data for the desired energy estimates to hold and for the main theorem of Hoff in \cite{Hoff} to be valid.
\begin{defn}\label{definition:admissible}
We say that a family of smooth functions $\{(\rho_0^\varepsilon,u_0^\varepsilon)\}_{0<\varepsilon\leq\eps_0}$ for some $\eps_0>0$ is an \emph{admissible sequence of initial data} if the following assumptions hold:
\begin{itemize}
\item The total relative mechanical energies are uniformly bounded, i.e.,
\begin{equation*}
\sup_{\varepsilon\in(0,\eps_0]} E[\rho_0^\varepsilon,u_0^\varepsilon] \leq E_0 < \infty;
\end{equation*}
\item For the initial densities, there holds the weighted derivative uniform bound
\begin{equation*}
\sup_{\varepsilon\in(0,\eps_0]} \varepsilon^2 \int_\mathbb{R} \frac{|\rho_{0,x}^\varepsilon(x)|^2}{\rho_0^\varepsilon(x)^3} \, dx \leq E_1 < \infty;
\end{equation*}
\item The relative total initial momenta are uniformly bounded, i.e.,
\begin{equation*}
\sup_{\varepsilon\in(0,\eps_0]} \int_\mathbb{R} \rho_0^\varepsilon(x) |u_0^\varepsilon(x) - \bar{u}(x)| \, dx \leq M_0 <\infty;
\end{equation*}
\item The initial densities are bounded away from the vacuum, i.e.,
\begin{equation*}
\text{there exists } c_0^\varepsilon > 0 \text{ such that } \rho_0^\varepsilon(x) \geq c_0^\varepsilon>0 \text{ for all } x \in \mathbb{R}, \text{ for each } \eps>0.
\end{equation*}
\end{itemize}
\end{defn}

Following similar proofs to those of \cite[Lemmas 3.2--3.4]{SS1}, making use of the results of Lemma \ref{lemma:pointwise hat chapter 2}, we obtain the following result.
\begin{prop}[Uniform estimates]\label{lemma:unif estimates ch2}
    Assume that $\{(\rho_0^\varepsilon,u_0^\varepsilon)\}_{0<\varepsilon\leq\eps_0}$ is an admissible sequence of initial data, in the sense of Definition \ref{definition:admissible}, and let $\{ (\rho^\varepsilon,u^\varepsilon) \}_{0<\varepsilon\leq\eps_0}$ be the sequence of viscous solutions of \eqref{eq:ns} that it generates. Then, for any $T>0$, there exist constants $M_1,M_2>0$, independent of $\eps$, such that, for all $\eps\in(0,\eps_0]$,
    \begin{itemize}
        \Item[(i)] 
        \begin{equation}
            \sup_{t \in [0,T]} E[\rho^\varepsilon,u^\varepsilon](t) + \int_0^T \int_\mathbb{R} \varepsilon|u^\varepsilon_x|^2 \, dx \, dt \leq M_1.
        \end{equation}
        \Item[(ii)] 
        \begin{equation}
            \varepsilon^2 \int_\mathbb{R} \frac{|\rho^\varepsilon(T,x)|^2}{\rho^\varepsilon(T,x)^3} \, dx + \varepsilon \int_0^T \int_\mathbb{R} \frac{p'(\rho^\varepsilon)}{(\rho^\varepsilon)^2}|\rho^\varepsilon_x|^2 \, dx \, dt \leq M_2.
        \end{equation}
        \end{itemize}
Moreover, for any compact $K\subset\R$, there exist constants $M_3,M_4>0$, depending on $T$ and $K$, but independent of $\eps$, such that, for all $\eps\in(0,\eps_0]$,
\begin{itemize}
        \Item[(iii)]
        \begin{equation}
            \int_0^T \int_K \rho^\varepsilon p(\rho^\varepsilon) \, dx \, dt \leq M_3.
        \end{equation}
        \Item[(iv)]
        \begin{equation}
            \int_0^T \int_K \rho^\varepsilon |u^\varepsilon|^3 \, dx \, dt \leq M_4.
        \end{equation}
    \end{itemize}
\end{prop}
The above key lemma, Lemma \ref{lemma:pointwise hat chapter 2}, is essential for the proof of the final higher integrability estimate of this Proposition (see \cite[Lemma 3.4]{SS1}).
Analogously, following the strategy of proof of \cite[Proposition 3.9]{SS1} while appealing to Lemma \ref{lemma:cpctly supported chapter 2}, we obtain the following compactness result.

\begin{prop}[Compactness of the entropy dissipation measures]\label{prop:cptness}
    Let $\psi \in C^2_c(\mathbb{R})$ and let $(\eta^\psi, q^\psi)$ be the associated entropy pair via \eqref{eq:generate entropy} and \eqref{eq:generate entropy flux}. Additionally, let $\{(\rho_0^\varepsilon,u_0^\varepsilon) \}_{\varepsilon>0}$ be a sequence of admissible initial data, in the sense of Definition \ref{definition:admissible}, and let $\{(\rho^\varepsilon,u^\varepsilon)\}_{\varepsilon>0}$ be the associated viscous solutions of \eqref{eq:ns}. Then, we have that the entropy dissipation measures
\begin{equation}
\eta^\psi(\rho^\varepsilon,\rho^\varepsilon u^\varepsilon)_t + q^\psi (\rho^\varepsilon,\rho^\varepsilon u^\varepsilon)_x
\end{equation}
are confined to a compact subset of $W^{-1,p}_{loc}(\mathbb{R}^2_+)$ for any $p \in (1,2)$.
\end{prop}

\section{Singularities of the entropy kernel}\label{section:singularities ch2}
The main result of this section is Lemma \ref{lemma:lambda plus one deriv chi}, which is the analogue of \cite[Lemma 2.7]{SS1}, and is concerned with the fractional derivatives of the kernels $\chi$ and $\sigma$ (\textit{cf.}~\cite[Section 2.2]{SS1}). This result is indispensable for the proof of the reduction of the Young measure, which is contained in the next section. We recall that for a distribution $g(s)$, the fractional derivative of order $\la$ of $g$ is defined by the following formula, where $\Gamma$ is the gamma function, 
$$\partial^\la g(s)=\Gamma(-\la)g(s)*[s]_+^{-\la-1}.$$
We require sharp bounds on the fractional derivatives of $\chi$ and $\sigma$. We begin with the expansions of these kernels from \cite[Theorems 2.1--2.3]{ChenLeFloch}.
\begin{thm}\label{lemma:global bound vacuum expansion for chi}
    The entropy kernel admits globally the expansion
    \begin{equation}\label{eq:expansion at vacuum}
    \chi(\rho,u) = a_\sharp(\rho)G_\lambda(\rho,u) + a_\flat(\rho)G_{\lambda+1}(\rho,u) + g_1(\rho,u) \qquad \text{for }(\rho,u) \in \mathbb{R}^2_+,
\end{equation}
where $G_\lambda(\rho,u) = [k(\rho)^2 - u^2]_+^\lambda$, and $g_1(\rho,\cdot)$ and its fractional derivative $\partial^{\lambda+1}_u g_1(\rho,\cdot)$ are H\"{o}lder continuous. There exists a constant $M>0$ such that
\begin{equation}\label{eq:a sharp and a flat large rho bound}
    |a_\sharp(\rho)| + |a_\flat(\rho)| \leq M \sqrt{\rho} k(\rho)^{-\lambda}  \qquad \text{for } \rho \geq \rho_*,
\end{equation}
and
\begin{equation}\label{eq:lambda+1 g1 far}
    \Vert \partial^{\lambda+1}_u g_1(\rho,\cdot)\Vert_{L^\infty(\mathbb{R})} \leq M  \sqrt{\rho} k(\rho)^2 \qquad \text{for } \rho \geq \rho_*.
\end{equation}
\end{thm}
\begin{proof}
Recall from \cite[Proposition 2.1]{ChenLeFloch} that
$a_\sharp(\rho) = \alpha_\sharp(\rho) k(\rho)^{-2\lambda-1}$, where
\begin{equation}
    \alpha_\sharp(\rho) = M_\lambda k(\rho)^{\lambda+1}k'(\rho)^{-1/2},
\end{equation}
while $a_\flat(\rho) = \alpha_\flat(\rho)k(\rho)^{-2\lambda-3}$, with
\begin{equation}
    \alpha_\flat(\rho) = -\frac{1}{4(\lambda+1)}k(\rho)^{\lambda+2}k'(\rho)^{-1/2} \int_0^\rho k(\tau)^{-(\lambda+1)}k'(\tau)^{-1/2} \alpha_\sharp''(\tau) \, d\tau.
\end{equation}
A calculation (\textit{cf.}~\cite[Lemma 3.1]{ChenLeFloch}) then shows that there exists a constant $M=M(\gamma,r)$ such that
\begin{equation}
    \rho^{-1}|\alpha_\sharp(\rho)| + |\alpha'_\sharp(\rho)| + |\alpha_\sharp''(\rho)| + \rho |\alpha_\sharp^{(3)}(\rho)| \leq M \qquad \text{for } \rho \in [0,r),
\end{equation}
where $r>0$ is as in Definition \ref{def:asymp isothermal}. From this, it follows by direct calculation that
\begin{equation}
    \rho^{-2}|\alpha_\flat(\rho)| + \rho^{-1}|\alpha_\flat'(\rho)| + |\alpha_\flat''(\rho)| \leq M \qquad \text{for } \rho \in [0,r).
\end{equation}
Moreover, calculating the derivatives explicitly and using the bounds on the derivatives of $k$ provided by \eqref{eq:k'' and higher, decay}, we obtain
\begin{equation}
    |\alpha_\sharp(\rho)| + \rho|\alpha_\sharp'(\rho)| + \rho^2|\alpha_\sharp''(\rho)| + \rho^3|\alpha_\sharp^{(3)}(\rho)|   \leq M \sqrt{\rho} k(\rho)^{\lambda+1}. \qquad \text{for } \rho \geq \rho_*,
\end{equation}
for some $M>0$ and also 
\begin{equation}
    |\alpha_\flat(\rho)| + \rho|\alpha_\flat'(\rho)| + \rho^2|\alpha_\flat''(\rho)|  \leq M \sqrt{\rho} k(\rho)^{\lambda+3} \qquad \text{for } \rho \geq  \rho_*,
\end{equation}
and therefore the bounds on $a_\sharp$ and $a_\flat$ in \eqref{eq:a sharp and a flat large rho bound} follow easily. 

For the remainder term, one may check  (\textit{cf.}~\cite[Proof of Theorems 2.1 and 2.2]{ChenLeFloch}) that
\begin{equation}
    \left\lbrace\begin{aligned}
        & g_{1,\rho \rho} - k'(\rho)^2 g_{1, u u} = A(\rho)k(\rho)^{-1} f_{\lambda+1}\big( \frac{u}{k(\rho)}\big), \\
        &g_1(0,\cdot) = 0, \\
        &g_{1,\rho}(0,\cdot) = 0,
    \end{aligned}\right.
\end{equation}
where $f_\lambda(y) = [1-y^2]^\lambda_+$ and
\begin{equation}
    A(\rho) = -\left(\alpha_\flat''(\rho) + \frac{2\lambda+3}{2(\lambda+1)}\alpha_\sharp''(\rho) \right) \qquad \text{for } \rho \geq 0,
\end{equation}
which is locally bounded. Thus there exists $M>0$ such that $|A(\rho)|\leq M$ for $\rho \in [0,\rho_*)$. Arguing as in the proof of Lemma \ref{lemma:second rep formula eta error}, we get the representation
\begin{equation*}
    \begin{aligned}
        k'(\rho_0) g_1 (\rho_0,u_0)  = \frac{1}{2\rho_0 } \int_0^{\rho_0} d(\rho) k'(\rho) \big( g_1(\rho,u+k(\rho_0)-k(\rho)) + g_1(\rho, u-k(\rho_0)+k(\rho)) \big) \, d\rho & \\
        + \frac{1}{2\rho_0}\int_0^{\rho_0} \rho A(\rho)k(\rho)^{-1} \bigg( \int_{u_0-k(\rho_0)+k(\rho)}^{u_0+k(\rho_0)-k(\rho)} f_{\lambda+1}\big( \frac{s}{k(\rho)} \big) \, ds \bigg) \, d\rho & .
    \end{aligned}
\end{equation*}
As $\rho A(\rho) k(\rho)^{-1}$ is bounded on $[0,\rho_*)$,  the second integral is well-defined. So, taking derivatives in $u_0$, we get
\begin{equation}\label{eq:rep formula partial u g1}
    \begin{aligned}
        k' (\rho_0) & \partial_u g_{1}(\rho_0,u_0)   \\ =&\,\frac{1}{2\rho_0 } \int_0^{\rho_0} d(\rho) k'(\rho) \big( \partial_u g_{1}(\rho,u_0+k(\rho_0)-k(\rho)) + \partial_u g_{1}(\rho, u_0-k(\rho_0)+k(\rho)) \big) \, d\rho \\
        &+ \frac{1}{2\rho_0}\int_0^{\rho_0} \rho A(\rho)k(\rho)^{-1}  \big( f_{\lambda+1}\big( \frac{u_0+k(\rho_0)-k(\rho)}{k(\rho)} \big) - f_{\lambda+1}\big( \frac{u_0-k(\rho_0)+k(\rho)}{k(\rho)} \big) \big) \, d\rho.
    \end{aligned}
\end{equation}
 By the fundamental theorem of calculus, we see that the fractional derivative $\partial^\la_s f_{\la+1}$ satisfies
\begin{equation}\label{eq:frac deriv ftc}
\begin{aligned}
    \partial^\lambda_s f_{\lambda+1}(s) - \partial^\lambda_s f_{\lambda+1}(-1) = \int^s_{-1} \partial^{\lambda+1}_y f_{\lambda+1}(y) \, dy.
\end{aligned}
\end{equation}
Note that $\supp f_{\lambda}=[-1,1]$ so that the second term on the left-hand side vanishes. Also, from \cite[Section 2.2]{SS1}, with $H$ the Heavyside function and $\ci$ the Cosine integral,
\begin{equation*}
    \partial^{\lambda+1}_s f_{\lambda+1}(s) = A_1^{\lambda+1} \big( H(s+1) + H(s-1) \big) + A_2^{\lambda+1} \big( \ci(s+1) - \ci(s-1) \big) + \tilde{r}^{\lambda+1}(s),
\end{equation*}
where $A^{\lambda+1}_1, A^{\lambda+1}_2 \in \mathbb{C}$ and $\tilde{r}^{\lambda+1}$ is a compactly supported H\"{o}lder continuous function. Hence, substituting this equality into \eqref{eq:frac deriv ftc}, we get 
\begin{equation}\label{eq:frac deriv lambda of f lambda+1}
\begin{aligned}
    \partial_s^\lambda f_{\lambda+1}(s) = A^{\lambda+1}_1 \big( [s+1]_+ + [s-1]_+ \big) &+ A^{\lambda+1}_2 \bigg( \int^s_{-1} \ci (y+1) \, dy - \int_{-1}^s \ci(y-1) \, dy \bigg) \\
    &+ \int_{-1}^s \tilde{r}^{\lambda+1}(y) \, dy.
    \end{aligned}
\end{equation}
Given the terms in \eqref{eq:frac deriv lambda of f lambda+1} and the support of $f_\lambda$, we see that $\partial_s^\lambda f_{\lambda+1}$ is a continuous function with compact support, and it is therefore uniformly bounded. In view of this, by applying the fractional derivative $\partial_{u_0}^{\lambda}$ to the representation formula \eqref{eq:rep formula partial u g1}, we get
\begin{equation}\label{eq:pre gronwall partial lambda plus one g}
\begin{aligned}
    \Vert k'(\rho_0) \partial^{\lambda+1}_u g_1(\rho_0,\cdot)\Vert_{L^\infty(\mathbb{R})} \leq \frac{1}{\rho_0}\int_0^{\rho_0} d(\rho)  \Vert k'(\rho) \partial^{\lambda+1}_u g_1 (\rho,\cdot)\Vert_{L^\infty(\mathbb{R})} \, d\rho & \\
    + \frac{M}{\rho_0}\int_0^{\rho_0} \rho A(\rho) k(\rho)^{-(\lambda+1)} \, d\rho &,
    \end{aligned}
\end{equation}
where we have taken into account the homogeneity of the fractional derivative in the final term. We now bound this final term. Suppose that $\rho_0 \in (0,r)$, then
\begin{equation*}
    \int_0^{\rho_0} |\rho A(\rho) k(\rho)^{-(\lambda+1)}| \, d\rho \leq M\int_0^{\rho_0} \rho^{\frac{1}{2}-\frac{\theta}{2}} \, d\rho \leq M \rho_0^{\frac{3}{2}-\frac{\theta}{2}},
\end{equation*}
and since $\frac{3}{2}-\frac{\theta}{2} > \frac{1}{2}$, the right-hand side is dominated by $M\sqrt{\rho_0}$. If $\rho_0 \in [r,\rho_*)$, only a bounded contribution is added, since the integrand is continuous on this interval. Finally, if $\rho_0 \geq \rho_*$,
\begin{equation*}
\begin{aligned}
    \int_0^{\rho_0} |\rho A(\rho) k(\rho)^{-(\lambda+1)}| \, d\rho &\leq M\int_0^{r} \rho^{\frac{1}{2}-\frac{\theta}{2}} \, d\rho + M \int_r^{\rho_*} \, d\rho + M \int_{\rho_*}^{\rho_0} k(\rho)^2 \rho^{-1/2} \, d\rho, \\
    &\leq M \big( 1 + \sqrt{\rho_0}k(\rho_0)^2\big).
\end{aligned}
\end{equation*}
Thus, applying Gr\"{o}nwall's lemma to \eqref{eq:pre gronwall partial lambda plus one g}, we obtain
\begin{equation}\label{eq:lambda+1 g1 near vac}
    \Vert k'(\rho_0) \partial^{\lambda+1}_u g_1(\rho_0,\cdot)\Vert_{L^\infty(\mathbb{R})} \leq M \rho_0^{\frac{1}{2}-\frac{\theta}{2}} \qquad \text{for } \rho_0 \in [0,r), 
\end{equation}
while
\begin{equation*}
    \Vert k'(\rho_0) \partial^{\lambda+1}_u  g_1(\rho_0,\cdot)\Vert_{L^\infty(\mathbb{R})} \leq \frac{M k(\rho_0)^2}{\sqrt{\rho_0}}  \qquad \text{for } \rho_0 \geq \rho_*,
\end{equation*}
from which the result is easily deduced using \eqref{eq:k'' close to rho squared}.
\end{proof}

A similar proof shows that we also have the equivalent results for the entropy flux kernel.

\begin{thm}\label{lemma:global bound vacuum expansion for h}
Recall from \cite[Theorem 2.3]{ChenLeFloch} that $h(\rho,u-s) = \sigma(\rho,u,s) - u\chi(\rho,u-s)$ admits globally the expansion
\begin{equation}\label{eq:expansion at vacuum flux}
    h(\rho,u) = -u\big( b_\sharp(\rho)G_\lambda(\rho,u) + b_\flat(\rho) G_{\lambda+1}(\rho,u) \big) + g_2(\rho,u) \qquad \text{for } (\rho,u) \in \mathbb{R}^2_+,
\end{equation}
where $G_\lambda(\rho,u) = [k(\rho)^2 - u^2]_+^\lambda$, and $g_2(\rho,\cdot)$ and its fractional derivative $\partial^{\lambda+1}_u g_2(\rho,\cdot)$ are H\"{o}lder continuous. There exists a constant $M>0$ such that
\begin{equation}
    |b_\sharp(\rho)| + |b_\flat(\rho)| \leq M\sqrt{\rho} k(\rho)^{-\lambda-1} \qquad \text{for } \rho \geq \rho_*,
\end{equation}
and
\begin{equation}
    \Vert \partial^{\lambda+1}_u g_2(\rho,\cdot) \Vert_{L^\infty(\mathbb{R})} \leq M\sqrt{\rho} k(\rho)^2 \qquad \text{for } \rho \geq \rho_*.
\end{equation}
\end{thm}

We recall the following lemma from \cite[Proposition 2.4]{ChenLeFloch} for future use.
\begin{lemma}\label{lemma:D prop}
The coefficients in the expansions \eqref{eq:expansion at vacuum} and \eqref{eq:expansion at vacuum flux} satisfy
\begin{equation}
D(\rho) := a_\sharp(\rho)b_\sharp(\rho) - k(\rho)^2 \big( a_\sharp(\rho) b_\flat(\rho) - a_\flat(\rho) b_\sharp(\rho) \big) > 0.
\end{equation}
\end{lemma}
We now arrive at the main result of this section, which will be used in key steps of the Young measure reduction in Section \ref{section:reduction ch2}.

\begin{lemma}\label{lemma:lambda plus one deriv chi}
    For all values of $(\rho,u,s) \in \mathbb{R}^2_+ \times \mathbb{R}$, the fractional derivatives $\partial_u^{\lambda+1}\chi$ and $\partial^{\lambda+1}_u h$ admit the expansions:
    \begin{equation}\label{eq:structure frac deriv chi}
    \begin{aligned}
        \partial_s^{\lambda+1}\chi(\rho & ,u-s)  \\
       = \sum_{\pm} \bigg( & A_{1,\pm} (\rho) \delta(s-u \pm k(\rho)) + A_{2,\pm}(\rho) H(s-u \pm k(\rho)) \\
        &+ A_{3,\pm}(\rho) \pv (s-u \pm k(\rho)) + A_{4,\pm}(\rho) \ci (s-u \pm k(\rho)) \bigg) + r_\chi (\rho,u-s),
    \end{aligned}
    \end{equation}
    and
     \begin{equation}\label{eq:structure frac deriv h}
    \begin{aligned}
        \partial_s^{\lambda+1}h(\rho  ,u-s) & \\
     =   \sum_{\pm} (s-u)\bigg( & B_{1,\pm} (\rho) \delta(s-u \pm k(\rho)) + B_{2,\pm}(\rho) H(s-u \pm k(\rho)) \\
        &+ B_{3,\pm}(\rho) \pv (s-u \pm k(\rho)) + B_{4,\pm}(\rho) \ci (s-u \pm k(\rho)) \bigg) \\
        +\sum_\pm \bigg( B_{5,\pm}  (&\rho) H(s-u\pm k(\rho)) + B_{6,\pm}(\rho) \ci (s-u \pm k(\rho)) \bigg) + r_\sigma(\rho,u-s),
    \end{aligned}
    \end{equation}
    where $\de$, $\pv$, $H$ and $\ci$ denote the Dirac mass, Principal Value, Heaviside function and Cosine integral, respectively, and the remainder terms $r_\chi$ and $r_\sigma$ are H\"{o}lder continuous functions. Additionally, in the large density regime, there exists a positive constant $M$ such that 
    \begin{equation}\label{eq:coeff bounds frac deriv chi}
       \sum_{j=1,\pm}^{4} |A_{j,\pm}(\rho)| + \sum_{j=1,\pm}^6 |B_{j,\pm}(\rho)| \leq M\sqrt{\rho}\big( 1 + \log(\rho/\rho_*)\big) \qquad \text{for } \rho \geq \rho_*,
    \end{equation}
    and
     \begin{equation}\label{eq:coeff bounds frac deriv chi remainder}
        \Vert r_\chi(\rho,\cdot) \Vert_{L^\infty(\mathbb{R})} + \Vert r_\sigma(\rho,\cdot) \Vert_{L^\infty(\mathbb{R})} \leq M\sqrt{\rho}\big( 1 + (\log(\rho/\rho_*))^2\big) \qquad \text{for } \rho \geq \rho_*.
    \end{equation}
\end{lemma}
\begin{proof}
By taking fractional derivatives of \eqref{eq:expansion at vacuum}, we have, as in \cite[Section 2.2]{SS1},
\begin{equation*}
    \partial_s^{\lambda+1} \chi(\rho,u-s) = a_\sharp(\rho) \partial^{\lambda+1}_s G_\lambda (\rho,u-s) + a_\flat(\rho)\partial^{\lambda+1}_s G_{\lambda+1}(\rho,u-s) + \partial^{\lambda+1}_s g_1 (\rho,u-s),
\end{equation*}
where we know from \cite[Proposition 3.4]{LeFlochWestdickenberg} that
\begin{equation*}
\begin{aligned}
    \partial^\lambda_s G_\lambda(\rho,u-s) = & \,k(\rho)^\lambda  A_1^\lambda \big( H(s-u+k(\rho)) + H(s-u-k(\rho)) \big) \\
    &+ k(\rho)^\lambda A^\lambda_2 \big( \ci(s-u+k(\rho)) - \ci(s-u-k(\rho)) \big) +k(\rho)^\lambda \tilde{r}\big( \frac{s-u}{k(\rho)} \big),
\end{aligned}
\end{equation*}
and
\begin{equation*}
    \begin{aligned}
        \partial_s^{\lambda+1}G_\lambda(\rho,u-s) = &\,k(\rho)^\lambda \big[  A_1^\lambda \big( \delta(s-u+k(\rho)) + \delta(s-u-k(\rho)) \big) \\
       & \qquad+ A^\lambda_2 \big( \pv(s-u+k(\rho)) - \pv(s-u-k(\rho)) \big) \big] \\
        &+k(\rho)^{\lambda-1} \big[  A^\lambda_3 \big( H(s-u+k(\rho)) - H(s-u-k(\rho)) \big) \\
        & \qquad+ A^\lambda_4 \big( \ci(s-u+k(\rho)) - \ci(s-u-k(\rho)) \big) \big] \\
        &+k(\rho)^{\lambda-1} \big( -A^\lambda_4 (\log k(\rho))^2 + \tilde{q} \big( \frac{s-u}{k(\rho)} \big) \big),
    \end{aligned}
\end{equation*}
where $A^\lambda_i \in \mathbb{C}$ for $i \in \{1,\dots,4\}$ are $\lambda$-dependent constants, and $\tilde{r}$ and $\tilde{q}$ are uniformly bounded H\"{o}lder continuous functions. We therefore deduce the form for $\partial^{\lambda+1}_u \chi$ given in \eqref{eq:structure frac deriv chi}, with
\begin{equation*}
    \begin{aligned}
         &A_{1,\pm}(\rho) = a_\sharp(\rho)k(\rho)^\lambda A_1^\lambda , \qquad && A_{2,\pm}(\rho) = \pm a_\sharp(\rho)k(\rho)^{\lambda-1} A^\lambda_3 + a_\flat(\rho)k(\rho)^{\lambda+1}A_1^{\lambda+1}, \\
         &A_{3,\pm}(\rho) = \pm a_\sharp(\rho)k(\rho)^\lambda A^\lambda_2, \qquad && A_{4,\pm}(\rho) = \pm a_\sharp(\rho)k(\rho)^{\lambda-1} A^\lambda_4 \pm a_\flat(\rho)k(\rho)^{\lambda+1}A^{\lambda+1}_2,
         \end{aligned}
         \end{equation*}
         and
         \begin{equation*}
         \begin{aligned}
         r_\chi(\rho,u-s) = a_\sharp(\rho)k(\rho)^{\lambda-1} \tilde{q} \big( \frac{s-u}{k(\rho)} \big) &+ a_\flat(\rho)k(\rho)^{\lambda+1} \tilde{r}\big( \frac{s-u}{k(\rho)} \big) \\
         &- k(\rho)^{\lambda-1} A^\lambda_4 (\log k(\rho))^2 + \partial^{\lambda+1}_s g_1 (\rho,u-s).
    \end{aligned}
\end{equation*}
The conclusion of the lemma now follows easily from Lemma \ref{lemma:global bound vacuum expansion for chi}, using the fact that $k(\rho) \leq M (1 + \log (\rho/\rho_*))$ for $\rho \geq \rho_*$, by Corollary \ref{corollary:k sim log}. An identical procedure using the estimates of Lemma \ref{lemma:global bound vacuum expansion for h} yields the result for the entropy flux kernel.
\end{proof}

\section{The Young measure and proof of the main result}\label{section:reduction ch2}

We begin this section by constructing a probability measure which characterises the limiting behaviour of the viscous solutions of the Navier--Stokes equations \eqref{eq:ns}: the Young measure. We follow the approach in \cite[Section 2.3]{LeFlochWestdickenberg}, which is also common to \cite[Section 5]{ChenPerep1} and \cite[Section 4]{SS1}.

Define the upper half-plane $\mathbb{H} := \{ (\rho,u)\in\mathbb{R}^2 : \rho > 0 \}$, and the ring of continuous functions
\begin{equation}
\bar{C}(\mathbb{H}) := \left\lbrace \phi \in C(\bar{\mathbb{H}}) \left| \begin{aligned}
& \phi(0,\cdot) \text{ is a constant function and such that the} \\
&\text{mapping } (\rho,u)\mapsto \lim_{s\to \infty} \phi(s\rho,su) \text{ lies in } C(\mathbb{S}^1 \cap \bar{\mathbb{H}})
\end{aligned} \right. \right\rbrace,
\end{equation}
where $\mathbb{S}^1$ is the unit circle. This particular subset of functions is chosen to mitigate difficulties at the vacuum and in the limit of infinitely large densities.

Since $\bar{C}(\mathbb{H})$ is a complete subring of the continuous functions on $\mathbb{H}$ containing the constant functions, there exists a compactification $\overline{\mathcal{H}}$ of $\mathbb{H}$ such that $C(\overline{\mathcal{H}}) \cong \bar{C}(\mathbb{H})$, where $\cong$ means isometrically isomorphic to. The topology on $\overline{\mathcal{H}}$ is the weak-star topology induced by $C(\overline{\mathcal{H}})$, i.e., a sequence $(v_n)_{n\in\mathbb{N}} \subset \overline{\mathcal{H}}$ converges to $v \in \overline{\mathcal{H}}$ if and only if $|\phi(v_n) - \phi(v)| \to 0$ for all functionals $\phi \in C(\overline{\mathcal{H}})$. This topology is both separable and metrizable (see \cite{Roubicek} for details). Additionally, by defining $V$ to be the weak-star closure of the set $\{ \rho = 0 \}$, one can show that $\mathcal{H} = \mathbb{H} \cup V$. We remark that the topology on $\overline{\mathcal{H}}$ has the useful property that it does not distinguish points in the vacuum set.

By applying the fundamental theorem of Young measures for maps into compact metric spaces (\cite[Theorem 2.4]{AlbertiMueller}) we find that, for some subsequence of the viscous solutions $(\rho^{\varepsilon'},u^{\varepsilon'})$, there exists a measure $\nu_{t,x} \in \mathrm{Prob}(\overline{\mathcal{H}})$ for almost every $(t,x) \in \mathbb{R}^2_+$ with the property that, given any $\phi \in C(\overline{\mathcal{H}})$,
\begin{equation*}
\phi(\rho^{\varepsilon'},u^{\varepsilon'}) \overset{\ast}{\rightharpoonup} \int_{\overline{\mathcal{H}}} \phi(\rho,u) \, d\nu_{t,x}(\rho,u) \quad \text{in } L^\infty(\mathbb{R}^2_+) \text{ as } \varepsilon' \to 0.
\end{equation*}
For ease of notation, we relabel this subsequence as $(\rho^\varepsilon,u^\varepsilon)$. One can in fact show that the set of test functions $\phi$ for which such convergence holds is larger than simply $C(\overline{\mathcal{H}})$. By following analogous arguments to those in the proof of \cite[Proposition 4.1]{SS1}, making use of Proposition \ref{lemma:unif estimates ch2}, we arrive at the following result. 

\begin{lemma}
Let $\nu_{t,x}$ be a Young measure associated to a sequence of solutions to \eqref{eq:ns} generated by admissible initial data, in the sense of Definition \ref{definition:admissible}. Then $\nu_{t,x}$ has the following properties:
\begin{enumerate}
\item For almost every $(t,x)$, the measure $\nu_{t,x}$ is a probability measure on $\mathcal{H}$, i.e., $\int_{\mathcal{H}} \, d\nu_{t,x}=1$;
\item The mapping $(t,x) \mapsto \int_\mathbb{H} (\rho p(\rho) + \rho |u|^3) \, d\nu_{t,x}(\rho,u)$ belongs to $ L^1_{loc}(\mathbb{R}^2_+)$;
\item The space of admissible test functions may be extended in the following way:
Let $\phi \in C(\bar{\mathbb{H}})$ be such that $\phi|_{\partial\mathbb{H}} = 0$. Suppose that there exists $a>0$ such that $\supp \phi \subset \{ u+k(\rho) \geq -a , u-k(\rho) \leq a \}$. Provided that $\phi$ also satisfies the growth condition 
\begin{equation}
\lim_{\rho \to \infty} \rho^{-2}|\phi(\rho,u)| =0, \text{ uniformly for } u \in \mathbb{R},
\end{equation}
then $\phi$ is integrable with respect to $\nu_{t,x}$ for almost every $(t,x)$ and we have the convergence 
\begin{equation*}
\phi(\rho^\varepsilon,u^\varepsilon) \rightharpoonup \int_\mathbb{H} \phi(\rho,u) \, d\nu_{t,x}(\rho,u) \quad \text{in } L^1_{loc}(\mathbb{R}^2_+).
\end{equation*}
\end{enumerate}
\end{lemma}

Likewise, by arguments analogous to those in the proof of \cite[Proposition 4.2]{SS1}, which entail an application of the div-curl lemma for sequences with divergence and curl compact in $W^{-1,1}$ due to Conti--Dolzmann--M\"uller, \cite{CDM}, we use Proposition \ref{prop:cptness} in conjunction with Lemma \ref{lemma:cpctly supported chapter 2} to obtain the following commutation relation, in the style of Tartar--Murat (\textit{cf.}~\cite{Tartar}).
\begin{prop}\label{lemma:commutation ch2}
Let $\{(\rho_0^\varepsilon,u_0^\varepsilon)\}_{0<\varepsilon\leq\eps_0}$ be an admissible sequence of initial data in the sense of Definition \ref{definition:admissible} and let $\{ (\rho^\varepsilon , u^\varepsilon) \}_{0<\varepsilon\leq\eps_0}$ be the associated viscous solutions of \eqref{eq:ns}. Correspondingly, let $\nu_{t,x}$ be the Young measure generated by the family $\{ (\rho^\varepsilon , u^\varepsilon) \}_{0<\varepsilon\leq\eps_0}$. Then, we have the constraint equation
\begin{equation}
\overline{\chi(s_1) \sigma(s_2) - \chi(s_2)\sigma(s_1)} = \overline{\chi(s_1)} \, \overline{\sigma(s_2)} - \overline{\chi(s_2)} \, \overline{\sigma(s_1)},
\end{equation}
for all $s_1,s_2 \in \mathbb{R}$, where we use the notation $\overline{f} = \int f \, d\nu_{t,x}$ and $\chi(s_j)=\chi(\cdot,\cdot-s_j)$, $\sigma(s_j)=\sigma(\cdot,\cdot,s_j)$.
\end{prop}
The main theorem of this section is then the following reduction result.
\begin{thm}\label{thm:reduction ch2}
Let $\nu \in \mathrm{Prob}(\mathcal{H})$ be such that $\int_\mathcal{H} \rho^2 \, d\nu(\rho,u) < \infty$ and, for all $s_1,s_2 \in \mathbb{R}$,
\begin{equation}\label{eq:commutation relation ch2}
\overline{\chi(s_1)\sigma(s_2) - \chi(s_2)\sigma(s_1)} = \overline{\chi(s_1)}\,\overline{\sigma(s_2)} - \overline{\chi(s_2)}\,\overline{\sigma(s_1)}.
\end{equation}
Then either $\supp \nu \subset V$, or the support of $\nu$ is a singleton in $\mathbb{H}$.
\end{thm}
The proof of this theorem follows from now standard arguments (\textit{cf.}~\cite{ChenPerep1,DiPerna2} and see especially \cite[Theorem 5.1]{SS1}) once we have shown the next lemma (Lemma \ref{lemma:technical1 ch2}).

In what follows, we consider the fractional derivative operators $P_j := \partial_{s_j}^{\lambda+1}$, for $j = 2,3$, where it is understood that the distributions $\overline{P_j \chi(s_j)}$ act on test functions $\psi \in \mathcal{D}(\mathbb{R})$ via 
\begin{equation*}
\langle \overline{P_j \chi(s_j)}, \psi \rangle = - \int_\mathbb{R} \overline{\partial^\lambda_{s_j} \chi(s_j)} \psi'(s_j) \, ds_j \qquad \text{for } j = 2,3.
\end{equation*}
Additionally, we let $\phi_2,\phi_3 \in \mathcal{D}(\mathbb{R})$ be standard mollifiers, with $\int_\mathbb{R} \phi_j (s_j) \, ds_j = 1$ and $\phi_j \geq 0$ for $j=2,3$, chosen such that
\begin{equation}\label{eq:Y def}
Y(\phi_2,\phi_3) := \int_{-\infty}^\infty \int_{-\infty}^{s_2} \big( \phi_2(s_2) \phi_3(s_3) - \phi_2(s_3) \phi_3(s_2) \big) \, ds_3 \, ds_2 > 0.
\end{equation}
 As is standard, for $\delta > 0$, we define $\phi^\delta_j(s_j) := \delta^{-1} \phi_j(s_j/\delta)$, and, for $j=2,3$,
 \begin{equation*}
 \overline{P_j \chi_j^\delta} := \overline{P_j \chi_j} * \phi_j^\delta(s_1) = \int_\mathbb{R} \overline{\partial^\lambda_{s_j}\chi(s_j)} \delta^{-2} \phi_j' \big( \frac{s_1-s_j}{\delta} \big) \, ds_j.
 \end{equation*}
\begin{lemma}\label{lemma:technical1 ch2}
For any test function $\psi \in \mathcal{D}(\mathbb{R})$, writing also $\chi_1=\chi(s_1)$,
\begin{itemize}
\Item[(i)]
\begin{equation*}
\begin{aligned}
\lim_{\delta \to 0} \int_\mathbb{R} \overline{\chi(s_1)} \, & \overline{P_2 \chi_2^\delta P_3 \sigma_3^\delta - P_3 \chi_3^\delta P_2 \sigma_2^\delta}(s_1) \psi(s_1) \, ds_1 \\
&= \int_{\mathcal{H}} Y(\phi_2,\phi_3) Z(\rho) \sum_{\pm} (K^\pm)^2 \overline{\chi(u \pm k(\rho))} \psi(u \pm k(\rho)) \, d\nu (\rho,u),
\end{aligned}
\end{equation*}
where $Z(\rho) = (\lambda+1) M_\lambda^{-2} k(\rho)^{2\lambda} D(\rho) > 0$ for $\rho > 0$, and $D(\rho)$ is as in Lemma \ref{lemma:D prop}.
\Item[(ii)] \begin{equation*}
    \lim_{\delta \to 0} \int_\mathbb{R} \overline{P_3 \chi^\delta_3} \, \overline{P_2 \chi_2^\delta \sigma_1 - \chi_1 P_2 \sigma_2^\delta} \psi(s_1) \, ds_1 = \lim_{\delta \to 0} \int_\mathbb{R} \overline{P_2 \chi_2^\delta} \, \overline{P_3 \chi_3^\delta \sigma_1 - \chi_1 P_3 \sigma_3^\delta} \psi(s_1) \, ds_1.
\end{equation*}
\end{itemize}
\end{lemma}
The positive constants $M_\lambda$ and $K^\pm$ of the previous lemma were introduced in \cite[Section 2]{ChenLeFloch}.

Given the results of Lemma \ref{lemma:lambda plus one deriv chi} and the estimates of Proposition \ref{lemma:unif estimates ch2}, the proof of this lemma proceeds in the same manner as \cite[Lemmas 5.2--5.3]{SS1}. For the convenience of the reader, we give a sketch of the proof of the first part of the lemma below.
\begin{proof}[Sketch of Lemma \ref{lemma:technical1 ch2}(i)]
Let $\psi(s_1) \in \mathcal{D}(\mathbb{R})$. We recall that (see \cite[Lemma 4.3]{ChenLeFloch} and \cite[Proposition 5.5]{SS1}) on sets on which $\rho$ is bounded,
    \begin{equation}\label{eq:bison}
        P_2 \chi_2^\delta P_3 \sigma_3^\delta - P_3 \chi_3^\delta P_2 \sigma_2^\delta \rightharpoonup Y(\phi_2,\phi_3) Z(\rho) \sum_\pm (K^\pm)^2 \delta_{s_1 = u \pm k(\rho)}
    \end{equation}
    as $\delta \to 0$ weakly-star in measures in $s_1$ and uniformly in $(\rho,u)$ on sets where $\rho$ is bounded, where $Y$ was defined in \eqref{eq:Y def}. Thus, 
\begin{equation*}
\begin{aligned}
 \lim_{\delta \to 0} \int_{-\infty}^\infty \overline{\chi(s_1)} (P_2 \chi_2^\delta P_3\sigma_3^\delta &- P_3 \chi_3^\delta P_2 \sigma_2^\delta ) \psi(s_1) \, ds_1 \\
 &= Y(\phi_2 , \phi_3) Z(\rho) \sum_\pm (K^\pm)^2 \overline{\chi(u \pm k(\rho))} \psi(u \pm k(\rho))
\end{aligned}
\end{equation*}
 pointwise for all $(\rho,u)\in\mathcal{H}$. 
To interchange the pointwise convergence with the Young measure $\nu$, we observe that the function $\rho^2\in L^1(\mathcal{H},\nu)$. We therefore prove that 
\begin{equation}\label{ineq:claim}
    \bigg| \int_\R\overline{\chi(s_1)} (P_2 \chi_2^\delta P_3 \sigma_3^\delta - P_3 \chi_3^\delta P_2\sigma_2^\delta) \psi(s_1) \, ds_1 \mathds{1}_{\rho > \rho_*} \bigg| \leq C(\rho^2 + 1),
\end{equation}
for some $C>0$ independent of $\rho$, $u$ and $\delta$ and hence apply Lebesgue's Dominated Convergence Theorem (and the local uniform convergence of \eqref{eq:bison}) to conclude the proof of the lemma.

In order to prove claim \eqref{ineq:claim}, we exploit the decomposition $\sigma=u\chi+h$ by writing
\begin{equation*}
    P_2 \chi_2^\delta P_3 \sigma_3^\delta - P_3 \chi_3^\delta P_2\sigma_2^\delta = P_2 \chi_2^\delta P_3(\sigma_3^\delta - u\chi_3^\delta) - P_3 \chi_3^\delta P_2 (\sigma_2^\delta - u \chi_2^\delta)=P_2 \chi_2^\delta P_3 h_3^\delta - P_3 \chi_3^\delta P_2h_2^\delta.
\end{equation*}
Using Lemma \ref{lemma:lambda plus one deriv chi}, we see that this product consists of a sum of terms of the form
\begin{equation*}
    A_{i, \pm (\rho)} B_{j , \pm}(\rho) T_2(s_2 - u \pm k(\rho)) T_3(s_3 - u \pm k(\rho)),
\end{equation*}
and
\begin{equation*}
    A_{i, \pm (\rho)} B_{j , \pm}(\rho) (s_2-s_3)T_2(s_2 - u \pm k(\rho)) T_3(s_3 - u \pm k(\rho)),
\end{equation*}
where $T_2,T_3 \in \{ \delta , \pv , H , \ci \}$, or terms with the same structure but where $T_2 \in \{ \delta , H , \pv, \ci , r_\chi \}$ and $T_3 \in \{ H , \ci , r_\sigma \}$ and likewise with $s_2$ and $s_3$ reversed.

A simple estimate for distributions of these types (see \cite[Lemmas 3.8 and 3.9]{LeFlochWestdickenberg} for a proof)  yields, for any pair $T_2 , T_3 \in \{ \delta , \pv , H , \ci \}$,
\begin{equation}
\begin{aligned}
    \bigg|\int_{-\infty}^\infty \overline{\chi(s_1)} \psi(s_1) \big( (s_2-s_3) T_2(s_2 - u \pm k(\rho)) T_3(s_3 - u \pm k(\rho))  \big) & * \phi_2^\delta * \phi_3^\delta  (s_1) \, ds_1 \bigg| \\
    & \leq C\Vert \overline{\chi}\psi \Vert_{C^{0,\alpha}(\mathbb{R})},
    \end{aligned}
\end{equation}
where we have noted that $s \mapsto \overline{\chi(s)}$ is H\"{o}lder continuous. Likewise,
\begin{equation*}
\begin{aligned}
 \bigg|\int_{-\infty}^\infty \overline{\chi(s_1)} \psi(s_1) \big( T_2(s_2 - u \pm k(\rho)) & T_3(s_3 - u \pm k(\rho))  \big) * \phi_2^\delta * \phi_3^\delta (s_1) \, ds_1 \bigg| \\
 &\leq C \Vert \overline{\chi}\psi \Vert_{C^{0,\alpha}(\mathbb{R})} \big( 1 + \Vert r_\chi \Vert_{C^{0,\alpha}_{s_1}(\overline{B_R})} + \Vert r_\sigma \Vert_{C^{0,\alpha}_{s_1}(\overline{B_R})} \big),
\end{aligned}
\end{equation*}
for $T_2 \in \{ \delta , H , \pv , \ci , r_\chi \}$ and $T_3 \in \{ H , \ci , r_\sigma \}$. In this case, $R>0$ is such that $\supp \psi \subset B_{R-2}(0)$ and the terms involving $r_\chi$ and $r_\sigma$ are contributed when one of $T_2,T_3 \in \{ r_\chi , r_\sigma \}$. Hence, Lemma \ref{lemma:lambda plus one deriv chi} gives
\begin{equation*}
\begin{aligned}
     \bigg| \int_{-\infty}^\infty & \overline{\chi(s_1)}  (P_2 \chi_2^\delta P_3 \sigma_3^\delta - P_3 \chi_3^\delta P_2 \sigma_2^\delta) \psi(s_1) \, ds_1  \bigg| \\
     &\leq C \max_{j,k,\pm} \{ |A_{j,\pm}B_{k,\pm}| , |A_{j,\pm}| \cdot \Vert r_\chi \Vert_{C^{0,\alpha}_{s_1}(\overline{B_R})} , |B_{j,\pm}| \cdot \Vert r_\sigma \Vert_{C^{0,\alpha}_{s_1}(\overline{B_R})} \} \leq C(\rho^2 + 1).
\end{aligned}
\end{equation*}
Thus we have concluded the desired convergence.
\end{proof}

Finally, we conclude this section with a proof of the main theorem.

\begin{proof}[Proof of Theorem \ref{thm:main ch2}]
Let $\tilde{\rho}_0^\varepsilon(x) := \max \{ \rho_0(x),\sqrt{\varepsilon} \}$. Now mollify $\tilde{\rho}_0^\varepsilon$ and $u_0$ suitably, such that we obtain an admissible sequence of initial data $(\rho_0^\varepsilon,u_0^\varepsilon)$, in the sense of Definition \ref{definition:admissible}. The sequence of smooth solutions $(\rho^\varepsilon,u^\varepsilon)$ of \eqref{eq:ns} corresponding to this initial data then generates a Young measure $\nu_{t,x}$, constrained by the Tartar--Murat commutation relation, according to Proposition \ref{lemma:commutation ch2}. An application of Theorem \ref{thm:reduction ch2} then yields that $\nu_{t,x}$ is either a point mass or is supported in the vacuum set $V$. In the phase-space coordinates $(\rho,m)$, where $m=\rho u$, this measure is a Dirac mass. As such, we write $\nu_{t,x} = \delta_{(\rho(t,x),m(t,x))}$, where $\rho$ and $m$ are measurable functions. In view of this, we deduce that the convergence of the subsequence $(\rho^\varepsilon,\rho^\varepsilon u^\varepsilon) \to (\rho,m)$ occurs in measure and hence (up to subsequence) almost everywhere in $\mathbb{R}^2_+$. This proves that $(\rho,m)$ is a weak solution of \eqref{eq:euler}. It remains to check that it is an entropy solution, in the  sense of Definition \ref{def:entropy sol}.

Since the weak entropies are taken to be $C^2$ functions, the almost everywhere convergence guarantees that $\overline{\eta^*}(\rho^{\varepsilon},\rho^{\varepsilon}u^{\varepsilon}) \to \overline{\eta^*}(\rho,m) \text{ a.e. } (t,x) \in \mathbb{R}^2_+$. In turn, Fatou's lemma yields
\begin{equation*}
\int_\mathbb{R} \overline{\eta^*}(\rho,m) \, dx \leq \liminf_{\varepsilon} \int_\mathbb{R} \overline{\eta^*}(\rho^{\varepsilon},\rho^{\varepsilon}u^{\varepsilon}) \, dx \text{ for almost every } t \geq 0.
\end{equation*} 
In view of Proposition \ref{lemma:unif estimates ch2}, the right-hand side is bounded independently of $\varepsilon$. We thereby deduce that there exists $M=M(E_0,t)$, monotonically increasing such that
\begin{equation*}
\int_\mathbb{R} \overline{\eta^*}(\rho,m) \, dx \leq M(E_0,t) \text{ for almost every } t \geq 0.
\end{equation*}
The pair $(\rho,m)$ is therefore of relative finite-energy. The last part of the Definition \ref{def:entropy sol} can be verified by following the argument presented in the final portion of \cite[Section 6]{SS1}.
\end{proof}

\appendix
\section{Estimates on Elementary Quantities}\label{section:appendix}
We state and prove here various auxiliary estimates that we use throughout the paper to compare and estimate quantities relating to the pressure $p(\rho)$ and the function $k(\rho)$.
We begin with the proof of Corollary \ref{corollary:k' and k'' compare with rho powers}.

\begin{proof}[Proof of Corollary \ref{corollary:k' and k'' compare with rho powers}]
Since $k'(\rho) = \sqrt{p'(\rho)}/\rho$, we have
\begin{equation*}
    \begin{aligned}
    \left| k'(\rho)^2 - \frac{\kappa_2}{\rho^2} \right| = \left|\frac{p'(\rho) - \kappa_2}{\rho^2}\right| &\leq \frac{1}{\rho^2} \int_\rho^\infty \left|p''(y)\right| \, dy = \frac{3C_p}{\alpha}\rho^{-\alpha-2},
    \end{aligned}
\end{equation*}
while
\begin{equation*}
   \begin{aligned}
   \left|k''(\rho) + \frac{\sqrt{\kappa_2}}{\rho^2}\right| = \left|\frac{p''(\rho)}{2\rho\sqrt{p'(\rho)}} + \frac{\sqrt{\kappa_2}-\sqrt{p'(\rho)}}{\rho^2}\right| &\leq \frac{3 C_p \rho^{-\alpha-2}}{\sqrt{2 \kappa_2}} + \frac{1}{\rho^2} \int_\rho^\infty \frac{|p''(y)|}{2\sqrt{p'(y)}} \, dy \\
   &\leq \frac{3 C_p}{\sqrt{2\kappa_2}} \left( 1 + \alpha^{-1} \right)\rho^{-\alpha-2}.
   \end{aligned}
\end{equation*}
For the third derivative,
\begin{equation*}
    k^{(3)}(\rho) - \frac{2\sqrt{\kappa_2}}{\rho^3} = \frac{p'''(\rho)}{2\rho \sqrt{p'(\rho)}} - \frac{p''(\rho)}{\rho^2 \sqrt{p'(\rho)}} - \frac{p''(\rho)^2}{4\rho p'(\rho)^{3/2}} + \frac{2}{\rho^3}\big(\sqrt{p'(\rho)}-\sqrt{\kappa_2}\big),
\end{equation*}
from which we obtain
\begin{equation*}
    \left| k^{(3)}(\rho) - \frac{2\sqrt{\kappa_2}}{\rho^3} \right| \leq M\rho^{-\alpha-3} + \frac{1}{\rho^3}\int^\infty_\rho \frac{|p''(y)|}{\sqrt{p'(y)}} \, dy.
\end{equation*}
In a similar vein, we have the following, from which the result is easily deduced,
\begin{equation*}
    \left| k^{(4)}(\rho) + \frac{6\sqrt{\kappa_2}}{\rho^4} \right| \leq M\rho^{-\alpha-4} + \frac{3}{\rho^4}\int_\rho^{\infty}\frac{|p''(y)|}{\sqrt{p'(y)}} \, dy.
\end{equation*}
\end{proof}

We now state several other similar estimates, the proofs of which can be found in \cite[Chapter 3]{thesis}.

\begin{cor}\label{corollary:k sim log}
Assume that $\rho_* \geq \max\{ R , (4C_p/\kappa_2)^{1/\alpha} \}$. Then, there exists a positive $M=M(k(\rho_*),\kappa_2)$ such that
\begin{equation}\label{eq:k between log rho over rho star}
    M^{-1}\big( 1 + \log(\rho/\rho_*) \big) \leq k(\rho) \leq M \big( 1 + \log(\rho/\rho_*) \big) \qquad \text{for } \rho \geq \rho_*.
\end{equation}
\end{cor}

\begin{lemma}\label{lemma:d and d star}
Assume that $\rho_* \geq \max\{ R, (4C_p/\kappa_2)^{1/\alpha} \}$. Then, 
\begin{equation}
   0 < d_*(\rho) -1 = \frac{\rho_*}{\rho} \qquad \text{and} \qquad |d(\rho)-d_*(\rho)| \leq M\rho^{-\alpha} \qquad \text{for all } \rho \geq \rho_*.
\end{equation}
for some positive constant $M=M(C_p,\kappa_2)$ independent of $\rho_*$. Additionally, it follows that
\begin{equation}\label{eq:bound on d}
    |d(\rho)-1| \leq 2 \qquad \text{for all } \rho \geq \rho_*.
\end{equation}
\end{lemma}

\begin{lemma}\label{lemma:k' product bound}
We have the equality
\begin{equation}
    \frac{k'(\rho_0)}{k_*'(\rho_0)}  \frac{k'_*(\rho)}{k'(\rho)} = \sqrt{\frac{p'(\rho_0)}{p'(\rho)}} \qquad \text{for all } R\leq \rho \leq \rho_0.
\end{equation}
As such, provided $\rho_* \geq \max\{ R, (4C_p/\kappa_2)^{1/\alpha} \}$, there is a positive $M=M(\alpha,C_p,\kappa_2)$ such that
\begin{equation}
    \left| 1- \frac{k'(\rho_0)}{k_*'(\rho_0)}  \frac{k'_*(\rho)}{k'(\rho)} \right| \leq M\rho^{-\alpha} \qquad \text{for all } \rho_* \leq \rho \leq \rho_0.
\end{equation}
In turn, 
\begin{equation}
    0 < \frac{k'(\rho_0)}{k_*'(\rho_0)}  \frac{k'_*(\rho)}{k'(\rho)} \leq 2 \qquad \text{for all } \rho_* \leq \rho \leq \rho_0.
\end{equation}
\end{lemma}

\begin{lemma}\label{lemma:k difference rho0 and rho}
Assume that $\rho_* \geq \max\{ R, (4C_p/\kappa_2)^{1/\alpha} \}$. Then, there exists a positive constant $M=M(\alpha,C_p,\kappa_2)$ such that
\begin{equation}
    \left| (k(\rho_0)-k(\rho)) - (k_*(\rho_0)-k_*(\rho)) \right| \leq M\rho^{-\alpha}\min\big(\frac{\rho_0-\rho}{\rho},1\big) \qquad \text{for } \rho_* \leq \rho \leq \rho_0.
\end{equation}
\end{lemma}

\begin{lemma}\label{lemma:k' k'star ratio}
Assume that $\rho_* \geq \max\{ R, (4C_p/\kappa_2)^{1/\alpha} \}$. Then, there exists a positive constant $M=M(\alpha,C_p,\kappa_2)$ such that
\begin{equation}
    \left| 1 - \frac{k'(\rho_0)}{k_*'(\rho_0)} \right| \leq M \rho_0^{-\alpha} \qquad \text{for } \rho_0 \geq \rho_*.
\end{equation}
\end{lemma}

Next is a result concerning the size of the supports of the re-scaled isothermal kernel and the perturbation, which follows immediately from \eqref{eq:k star def} and \eqref{eq:k between log rho over rho star}.

\begin{lemma}\label{lemma:bound on max k and k star}
    There exists a positive constant $M=M(\alpha,\kappa_2,C_p,\rho_*)$ such that
    \begin{equation}
        \max\{ k(\rho),k_*(\rho) \} \leq M k(\rho)  \qquad \text{for all } \rho \geq \rho_*.
    \end{equation}
\end{lemma}

Finally, we state two technical lemmas regarding the relative internal energy.

\begin{lemma}\label{lemma: rho log rho bounded by e star}
    There exists a positive constant $M=M(\gamma, \kappa_1, \rho_*,k(\rho_*), \bar{\rho})$ such that
    \begin{equation}
       \rho + \rho \log (\rho/\rho_*) \leq M\big( 1 + e^*(\rho,\bar{\rho}) \big) \qquad \text{for } \rho \geq \rho_*.
    \end{equation}
\end{lemma}

\begin{lemma}\label{lemma:e star control chapter 2}
There exists a positive constant $M= M(\gamma, \kappa_1, \rho_* , k(\rho_*) , \bar{\rho})$, such that 
\begin{equation}\label{eq:e star control ch2}
0 \leq \rho + p(\rho) \leq M \big( 1 + e^*(\rho,\bar{\rho}) \big) \qquad \text{for } \rho \geq 0.
\end{equation}
\end{lemma}

The proofs of Lemmas \ref{lemma: rho log rho bounded by e star} and \ref{lemma:e star control chapter 2} have a common strategy, and one verifies \eqref{eq:eta m hat difference squared bounded by e star} using similar techniques. Complete proofs may be found in \cite[Chapter 3]{thesis}.

\textbf{Acknowledgement.} The authors wish to thank Gui-Qiang Chen for useful discussions. The second author was supported by Engineering and Physical Sciences Research Council [EP/L015811/1].

\end{document}